\documentclass[11pt]{amsart}
\usepackage{amsfonts,epsfig,fancyhdr}
\usepackage{newlfont,amsfonts,amssymb,amsmath,amsthm,amsgen,amscd,datetime,dsfont,hyperref,tensor}
\usepackage[small,nohug,heads=littlevee]{diagrams}
\usepackage[normalem]{ulem}
\diagramstyle[labelstyle=\scriptstyle]
\usepackage{multicol,picinpar,enumerate,cleveref,mathabx}
\usepackage{euscript,color}

\newcommand\re[1]{(\ref{#1})}

\newcommand{\ot}{\otimes}

\renewcommand{\L}{\mathbf{L}}
\renewcommand{\P}{\mathbf{P}}

\newcommand{\eeqref}[1]{equation~(\ref{#1})}

\newcommand{\C}{\ensuremath{\mathds{C}}}

\newcommand{\bmat}{\begin{pmatrix}}
\newcommand{\emat}{\end{pmatrix}}

\newcommand{\1}{\mathbf{1}}

\newcommand{\e}{\mathrm{e}}
\newcommand{\ev}{\mathbf{e}}
\renewcommand{\d}{{\mathrm d}}
\newcommand{\bcase}{\begin{case}}
\newcommand{\ecase}{\end{case}}

\newcommand{\bclaim}{\begin{claim}}
\newcommand{\eclaim}{\end{claim}}

\newcommand{\bstep}{\begin{step}}
\newcommand{\estep}{\end{step}}

\newcommand{\bhlem}{\begin{hlem}}
\newcommand{\ehlem}{\end{hlem}}

\newcommand{\bleer}{\begin{leer}}
\newcommand{\eleer}{\end{leer}}
\newcommand{\bde}{\begin{definition}}
\newcommand{\ede}{\end{definition}}

\newcommand{\bs}{\begin{proposition}}
\newcommand{\es}{\end{proposition}}
\newcommand{\btheo}{\begin{theorem}}
\newcommand{\etheo}{\end{theorem}}
\newcommand{\bfolg}{\begin{corollary}}
\newcommand{\efolg}{\end{corollary}}
\newcommand{\blem}{\begin{lemma}}
\newcommand{\elem}{\end{lemma}}
\newcommand{\bnote}{\begin{note}}
\newcommand{\enote}{\end{note}}
\newcommand{\bprf}{\begin{proof}}
\newcommand{\eprf}{\end{proof}}
\newcommand{\bd}{\begin{displaymath}}
\newcommand{\ed}{\end{displaymath}}
\newcommand{\be}{\begin{eqnarray*}}
\newcommand{\ee}{\end{eqnarray*}}
\newcommand{\eeqa}{\end{eqnarray}}
\newcommand{\beqa}{\begin{eqnarray}}
\newcommand{\bi}{\begin{itemize}}
\newcommand{\ei}{\end{itemize}}
\newcommand{\bnum}{\begin{enumerate}}
\newcommand{\enum}{\end{enumerate}}

\newcommand{\beq}{\begin{equation}}
\newcommand{\eeq}{\end{equation}}

\newcommand{\rr}{\mathds{R}}

\newcommand{\M}{M}

\newcommand{\vf}{\varphi}
\newcommand{\earr}{\end{array}\]}
\newcommand{\barr}{\[\begin{array}}
\newcommand{\bvec}{\left(\begin{array}{c}}
\newcommand{\evec}{\end{array}\right)}

\newcommand{\g}{\mathfrak{g}}
\renewcommand{\a}{\mathfrak{a}}
\renewcommand{\b}{\mathfrak{b}}

\newcommand{\h}{\mathfrak{h}}
\newcommand{\z}{\mathfrak{z}}

\renewcommand{\k}{\mathfrak{k}}
\renewcommand{\sl}{\mathfrak{sl}}
\newcommand{\hol}{\mathfrak{hol}}

\newcommand{\+}{\oplus}

\newcommand{\rrn}{\mathds{R}^n}
\newcommand{\so}{\mathfrak{so}}
\newcommand{\spin}{\mathfrak{spin}}

\newcommand{\gl}{\mathfrak{gl}}

\renewcommand{\sp}{\mathfrak{sp}}
\newcommand{\su}{\mathfrak{su}}
\newcommand{\SU}{\mathbf{SU}}

\newcommand{\SL}{\mathbf{SL}}

\newcommand{\del}{\partial}

\newcommand{\bbem}{\begin{bem}}
\newcommand{\ebem}{\end{bem}}
\newcommand{\bbez}{\begin{bez}}
\newcommand{\ebez}{\end{bez}}
\newcommand{\bbsp}{\begin{bsp}}
\newcommand{\ebsp}{\end{bsp}}
\newcommand{\pr}{\mathrm{pr}}

\newcommand{\im}{\mathrm{im}}

\newcommand{\trace}{\mathrm{tr }}
\DeclareMathOperator{\grad}{\mathrm{grad}}
\newcommand{\wt}{\widetilde}
\newcommand{\tnab}{\wt{\nabla}}

\newcommand{\tem}{\widetilde{M}}
\newcommand{\hm}{\widehat{\M}}
\newcommand{\hnab}{\widehat{\nabla}}
\newcommand{\hR}{\widehat{\RR}}

\newcommand{\tg}{\widetilde{g}}
\newcommand{\tgg}{\widetilde{\g}}

\DeclareMathOperator{\Span}{\mathrm{span}}

\renewcommand{\gg}{g}
\newcommand{\hg}{\widehat{\gg}}

\newcommand{\RR}{{R}}
\newcommand{\R}{R}
\newcommand{\hRic}{\widehat{Ric}}

\newcommand{\Id}{\mathrm{Id}}

\newcommand{\inter}{\makebox[7pt]{\rule{6pt}{.3pt}\rule{.3pt}{5pt}}\,}

\newcommand{\belabel}[1]{\begin{equation}\label{#1}}

\theoremstyle{definition}
\newtheorem{definition}{Definition}[section]
\newtheorem{bem}[definition]{Remark}
\newtheorem{bez}[definition]{Notation}
\newtheorem{bsp}[definition]{Example}

\newtheorem*{bsp*}{Example}
\newtheorem*{def*}{Definition}
\theoremstyle{plain}
\newtheorem{lemma}[definition]{Lemma}
\newtheorem*{lem*}{Lemma}
\newtheorem{proposition}[definition]{Proposition}
\newtheorem{corollary}[definition]{Corollary}
\newtheorem{theorem}[definition]{Theorem}

\numberwithin{equation}{section}

\begin{document}

\title
{Geometry and holonomy of indecomposable cones}

\thanks{This work was supported by 
the Australian Research
Council via the grants FT110100429 and DP120104582 and 
 by the German Science Foundation (DFG) under 
the Research Training Group 1670 and under Germany's Excellence Strategy -- EXC 2121 ``Quantum Universe'' -- 390833306. D.A.~is supported by grant n 18-00496 S of the Czech Science Foundation.
V.C.~is grateful to the University of Adelaide for its hospitality and support. V.C and T.L. 
thank  the mathematical research institute MATRIX in Australia where the first version of the paper was completed.
 }



\author{Dmitri Alekseevsky}\address{Institute for Information Transmission Problems, B.~Karetnuj per., 19, 127951, Moscow, Russia
     and University of Hradec Kr\'{a}lov\'{e}, Faculty of Science, Rokitansk\'{e}ho 62, 500 03 Hradec Kr\'{a}lov\'{e}, Czech Republic}
\email{dalekseevsky@iitp.ru}
\author{Vicente Cort\'{e}s}
\address{Department
   Mathematik, Universit\"at Hamburg, Bundesstra{\ss}e 55, D-20146
   Hamburg, Germany}\email{vicente.cortes@uni-hamburg.de}
\author{Thomas Leistner}\address{School of Mathematical Sciences, University of Adelaide, SA 5005, Australia}\email{thomas.leistner@adelaide.edu.au}

\begin{abstract}
We study the geometry and  holonomy of  semi-Riemannian, time-like 
metric cones that are indecomposable, i.e., which do not admit a local decomposition into a semi-Riemannian product.   
This includes irreducible cones, for which the holonomy can be classified, as well as non irreducible cones. The latter admit a parallel distribution of null $k$-planes, and we study the cases $k=1$ and $k=2$ in detail. In these cases, i.e., when the cone admits a distribution of parallel null tangent lines or planes, we give structure theorems about the base manifold. Moreover, in the case $k=1$ and when the base manifold is Lorentzian, we derive a description of the cone holonomy. This result is obtained by a computation of certain cocycles of indecomposable subalgebras in $\so(1,n-1)$.
 \end{abstract}

\maketitle
\tableofcontents
\section{Introduction}
\subsection{Background}
Cone constructions are a valuable tool in differential geometry to study overdetermined PDEs on manifolds. 
They are applied  in conformal \cite{fefferman/graham85,fefferman-graham07} and projective geometry \cite{Thomas25,Armstrong08-2}, but 
the most striking example is B\"{a}r's classification of Riemannian manifolds with real Killing spinors \cite{baer93}. B\"{a}r's observation that real Killing spinors on a Riemannian manifold $(M,g)$ correspond to parallel spinors on the cone
\[(\widecheck{M}=\rr^{>0}\times M, \widecheck{g}=dr^2+r^2 g),\]
allows to relate and apply several  fundamental results  in differential geometry: Berger's  list of irreducible Riemannian holonomy groups \cite{berger55} and the classification of those that belong to manifolds with parallel spinors by Wang \cite{wang89}, the understanding of the geometric structures that correspond to these holonomy groups, and finally Gallot's Theorem 
\cite{gallot79} 
that the cone $(\widecheck{M}, \widecheck{g})$ over a complete manifold $(M,g)$ is either flat or irreducible. This result allows to determine the geometry of $(M,g)$: if the cone $(\widecheck{M}, \widecheck{g})$ is flat, then $(M,g)$ has constant sectional curvature $1$, and if the cone is irreducible, the geometry of $(M,g)$ is determined by the special holonomy of the cone  (Ricci-flat K\"ahler, hyper-K\"ahler, or exceptional).

One of the motivations to study semi-Riemannian cones  is the Killing spinor equation on semi-Riemannian manifolds, but  indefinite cones already  become relevant in the Riemannian context. Indeed, {\em imaginary}  Killing spinors on a Riemannian manifold $(M,g)$ correspond to parallel spinors on the time-like cone
\begin{equation}\label{cone}(\hm=\rr^{>0}\times M, \hg=- \d r^2+r^2 g).\end{equation}
Riemannian manifolds with imaginary Killing spinors were classified by Baum in \cite{baum89-2, baum89-3} without using the cone construction, but our results about Lorentzian cones in \cite{acgl07} allow to reprove Baum's classification.

 Another motivation stems from supergravity (and string theory), where semi-Riemannian cones play a two-fold role. One the one hand, 
they appear as scalar geometries (of arbitrary dimension) in the superconformal formulation of supergravity theories, on the other hand, 
they can be used to study space-times which are part of supersymmetric solutions of the equations of motion of theories of (Poincar\'e) supergravity or of string theories.  
In the latter case, the supersymmetry equations can be analysed by passing to the time-like cone over the Lorentzian space-time manifold, which is a semi-Riemannian cone of index $2$.

A  generalisation of B\"ar's method to {\em indefinite} semi-Riemannian manifolds has two aspects: a holonomy classification of indefinite semi-Riemannian cones and the description of the corresponding  geometry of the base. Both tasks face  several difficulties in the  semi-Riemannian context.
The fundamental difficulty is  
that for metrics of arbitrary signature the holonomy group may not act completely reducibly: there are semi-Riemannian manifolds whose holonomy group admits an invariant subspace that is degenerate for the metric. As a consequence, those manifolds cannot be decomposed into a product of manifolds  with irreducible holonomy, as it is the case for Riemannian  manifolds.  Hence, 
in an indefinite semi-Riemannian context, irreducibility has to be replaced by indecomposability. A semi-Riemannian manifold is {\em indecomposable} if its holonomy representation (i.e., the representation of the holonomy algebra on the tangent space) does not admit an invariant subspace that is non-degenerate for the metric. By the splitting theorems of de Rham \cite{derham52} and Wu \cite{wu64}, such metrics do not have a local decomposition into product metrics, hence the term {\em indecomposable}. 
Therefore, a generalisation of B\"{a}r's method to semi-Riemannian geometry requires two steps:
\bi
\item[(A)] Generalise Gallot's Theorem  to the case of semi-Riemannian cones.
\item [(B)]
For indecomposable semi-Riemannian cones, 
describe the holonomy of the cone and the local geometry of the base. 
\ei
The problem in (A) was solved in \cite{acgl07}, where we studied decomposable indefinite semi-Riemannian cones and obtained a generalisation of Gallot's result.  In fact we showed that a cone over a complete and compact semi-Riemannian manifold is either flat or indecomposable.  The results in \cite{acgl07} have been  generalised in the compact and in the complete case in   \cite{matveev-mounoud10,matveev10,Mounoud12}.
Further results about decomposable cones have been obtained in 
 \cite[Theorems 5 and 6]{FedorovaMatveev14}. Cones  over Lorentzian Sasaki manifolds and their holonomy were studied in the decomposable and indecomposable case in    \cite{galaev06thesis}.


\subsection{Results}
In this article we deal with problem (B), i.e., we study the local geometry of the base and the holonomy of the cone in the case when the cone is {\em indecomposable}. This setting naturally splits into two different scenarios: the holonomy of the cone is {\em irreducible}, or it admits an {\em invariant subspace that is totally null but no non-degenerate invariant subspace}. The irreducible case is well understood as there is Berger's classification of irreducible holonomy groups \cite{berger55}, which we describe in Section \ref{irrsec} with the following result:
\btheo
\label{irredtheointro}
If $(\hm,\hg)$ is a time-like cone with irreducible holonomy algebra $\mathfrak{g}$, then $\g$ is isomorphic to one of the following Lie algebras
\begin{equation}\label{berger0}
\begin{array}{rclrclrcl}
&&\so(t,s),
&
\mathfrak{u}(p,q),\ \su(p,q) &\subset &  \so(2p,2q),
&
 \sp(p,q) &\subset & \so(4p,4q),
 \\
\so(n,\C)&\subset & \so(n,n),
& 
\g_2^\C&\subset& \so(7,7),
& \spin(7,\C)&\subset&\so(8,8),
 \\
 &&&
\g_2&\subset &\so(7), 
& 
\spin(7) &\subset & \so(8),
\\
&&&
\g_{2(2)} &\subset & \so(3,4),
&
\spin(3,4) &\subset &  \so(4,4).
\end{array}\end{equation}

\etheo

More interesting is the 
 non-irreducible indecomposable case. Here the cone admits a totally null vector distribution of rank $k>0$ that is invariant under parallel transport, or equivalently, its space of sections is invariant under differentiation with respect to the Levi-Civita connection. 
In general, this case is rather difficult and no general holonomy classification is known. However, 
the  parallel vector distribution determines the local structure of the base. This became obvious in   \cite{acgl07} where we studied the case of  \emph{Lorentzian} indecomposable cones. 
As mentioned, some of our motivation  comes from 
the equations of motion of supersymmetric theories of gravity, where the space-time metric is Lorentzian (that is of index $1$).  Hence we will focus on cones that have index $2$, that is signature $(2,n-2)$. For these the totally null parallel vector distribution is of rank $1$, i.e., a {\em null line}, or of rank  $2$, i.e., a {\em null plane}. Many of our results will however hold for cones in {\em arbitrary signature} but with an invariant  null line or  null plane.

 In Section \ref{localsec} we will study the case of a {\em parallel null line}, and  describe the local structure of the base as well as of the cone:

 \btheo\label{localtheointro}
 Let $(\hm,\hg)$ be the time-like cone  over a semi-Riemannian manifold $(M,g)$. 
 If  the cone admits a parallel null line field $\L$, then locally there is a parallel trivializing section of $\mathbf{L}$. Moreover, on a dense open subset $\hm_{\mathrm{reg}} \subset \hm$, the metric 
 $\hg$ is   locally  isometric to a warped product of the form 
\begin{equation}\label{ncone}\tg_0=2\, \d u\, \d v+u^2\gg_0,\end{equation}
with a semi-Riemannian metric $\gg_0$, and the metric $\gg$ is locally of the form
\[\gg=\d s^2+\mathrm{e}^{2s}\gg_0.\]
\etheo
In the case when the above decompositions hold globally (see Theorem \ref{glob:thm}), the situation can be summarised in the commutative diagram:
\begin{equation}\label{diagram}
\begin{diagram}
\text{signature }(t+1,s)\qquad&&&(\hm,\hg)&\rTo^{\text{isometry }\psi \quad}& (\widetilde{M},\tg_0)
\\
&& \ruTo_{\text{cone}}^{\hg=-\d r^2+r^2 \gg}&&\ruTo(2,4)^{\text{double warp}}_{\tg_0=2 \d u\d v+u^2\gg_0}&
\\
\text{signature }(t,s)\qquad&(M,g)&&&&
\\
&&\luTo^{\text{warp}}_{\gg=\d s^2+\mathrm{e}^{2s}\gg_0}&& & \\
\text{signature }(t,s-1)\qquad &&& (M_0,g_0)&&
\end{diagram}\end{equation} Here $\widetilde{M} = \mathbb{R}^+ \times \mathbb{R}^- \times M_0$, see (\ref{psi:eq}) for the definition of $\psi$. 
This result motivates the study of metrics of the form \re{ncone} in  \Cref{uvgsec}.  Such metrics have a parallel null vector field $\partial_v$ and it was shown in \cite{leistner05a} that their holonomy algebra  $\wt\g=\hol(\tg_0)$ is contained in $\hol (g_0)\ltimes \rr^{t,s}$, where $(t,s)$ is the signature of  the metric $g_0$, and moreover that $\pr_{\so(n)}(\wt\g)=\hol (g_0)$. For a Lorentzian metric $\tg$, i.e., when $g_0$ is Riemannian, it was shown in \cite{leistner05a,acgl07} that we have in fact \[\wt\g=\hol(g_0)\ltimes \rrn,\]
which means that the holonomy of the cone is determined solely by the holonomy of the metric $g_0$. In higher signatures, i.e., when $g_0$ is not Riemannian, this is no longer true, as examples will show. Our approach is to consider  the ideal of translations in $\hol(\tg_0)$,
\[T:=\hol(\tg_0)\cap \rr^{t,s},\]
and use this for a first, purely algebraic study of indecomposable subalgebras in the stabiliser of a null vector. This will be carried out in Section \ref{alg1sec}, which is the most technical section of the paper. The key observation is that 
\[\widetilde \g/T =\{ (X, \vf(X))\mid X\in \hol(g_0) \},\quad\text{ with }\vf\in Z^1(\hol(g_0), \rr^{t,s}/T),\]
where $Z^1(\hol(g_0), \rr^{t,s}/T)$ denotes the cocycles of $\hol(g_0)$ with values in $\rr^{t,s}/T$.
For example, in order to obtain results for time-like cones over Lorentzian manifolds, we will compute $ Z^1(\g, \rr^{1,n-1}/T)$, for   indecomposable subalgebras $\g$ of $\so(1,n-1)$ (these belong to one of four types according to \cite{bb-ike93}).

In Section \ref{holsec} we apply these algebraic results to obtain the following result. 
\btheo\label{holtheointro}
Let $\gg_0$ be a Lorentzian metric on an $n$-dimensional simply connected manifold $M$ and   
$\tg_0$ the metric of signature $(2,n)$  on $\rr^+\times \rr\times M$ defined in \re{ncone}.
  If the holonomy of $\tg_0$ acts indecomposably and with invariant null line, then 
\[\hol(\tg_0)=\hol(\gg_0)\ltimes \rr^{1,n-1},\]
or $g_0$ admits a parallel null vector field and $\tg_0$ admits two linearly independent parallel null vector fields that are orthogonal to each other.
\etheo
 This theorem shows that if the holonomy of $\tg_0$ is not equal to the semi-direct product $\hol(g_0)\ltimes \rr^{1,n-1}$, then $\tg_0$ and hence the cone admits a parallel null plane (which in addition is spanned by two parallel null vector fields). We study the case of cones admitting a totally null parallel $2$-plane in the remainder of the article. In Section \ref{planesec} we show:
 \btheo\label{planetheorintro}
If the timelike cone $(\hm,\hg)$ over a semi-Riemannian manifold $(M,g)$ admits a parallel, totally  null $2$-plane field, then, locally  over an open dense subset
the base $(\M,\gg)$  admits two vector fields $V$ and $Z$ satisfying
\begin{equation}\label{gVZint}
\gg(V,V)=0,\ \ \gg(Z,Z)=1,\ \ \gg(V,Z)=0,\end{equation}
and such that
 \begin{eqnarray}
  \label{nabVint}
  \nabla_XV
  &
  =
  &\alpha (X)V + \gg(X,V)Z,
  \\
  \label{nabZint}
    \nabla_XZ
  &=&
  -X +\beta(X)V+\gg(X,Z)Z,
  \end{eqnarray}
 with $1$-forms $\alpha $ and $\beta$ on $\M$.
 In particular,  the base $(\M,\gg)$ admits a geodesic, shearfree null congruence defined by $V$.

 Conversely, each pair of vector fields $V$ and $Z$ on $M$ satisfying relations (\ref{gVZint}), (\ref{nabVint}) and (\ref{nabZint}) defines a parallel distribution of totally null $2$-planes on the cone.

 \etheo
Note that equation \re{nabVint} implies that $V^\perp$ is integrable. This allows us to  determine the local form of the metrics with vector field $V$ and $Z$ satisfying equations (\ref{gVZint}--\ref{nabZint}):

\btheo 
A semi-Riemannian metric $(M,g)$ admits vector fields $V$ and $Z$ with (\ref{gVZint}--\ref{nabZint}) if and only if $(M,g)$ is locally of the form 
$M = M_0 \times \rr^3$ and 
\[ g = \d s^2 + e^{-2s}g_0(u) + 2\, \d u\,\eta,\]
for a family of metrics $g_0(u)$ on $M_0$ depending on $u$ and 
a $1$-form $\eta$ on $M$ such that $\eta (\partial_t)$ is nowhere vanishing 
satisfying the following system of first order PDEs: 
 \begin{equation}\begin{array}{rcl}
  \partial_t\eta_t=\partial_s\eta_t\ =\ X\eta_t&=&0,
  \\
  \partial_t\eta_s&  =&  2\eta_t,
  \\
  \partial_t(\eta(X))&=&0,\\
  \partial_s\,\eta (X) -X\ \eta_s&=&-2\eta (X)
   \label{system:equ-int}\end{array}\end{equation}
 for all $X\in \mathfrak{X}(M_0)$ and where $\eta_t:=\eta(\partial_t)$ and $\eta_s:=\eta(\partial_s)$.
\etheo
Finally we give explicitly the general solution for the system \re{system:equ-int}, providing us with a construction method of metrics whose cone admits a totally null two plane.

\subsection*{Acknowledgements} We would like to thank the anonymous referee for many valuable comments, in particular about the proof of Theorem \ref{z1lemma}, and for pointing out to us the result of Corollary \ref{refcor}.

\section{Preliminaries}
\subsection{Fundamental properties of time-like cones}
Let $(\M,\gg)$ be a semi-Riemannnian manifold and $\hm:=\rr^+\times \M$  with the metric
\belabel
{conemetric}
\hg:= -  \d r^2 +r^2\gg
\eeq
be the {\em time-like cone} or just the {\em cone over $(\M,\gg)$}. We denote by 
\[\xi=r\frac{\del}{\del r}
\] the {\em Euler vector field}.
The Levi-Civita connection $\hnab$ of $\hg$ reduces to the Levi-Civita connection $\nabla$ of $\gg$ in the following way 
\belabel{coneLC}
\begin{array}{rcl}
\hnab \xi  &=& \Id,\\
\hnab_X Y &=& \nabla_X Y +\gg(X,Y)\xi, 
\end{array}
\eeq
where here and in the following formulas $X,Y,Z\in \mathfrak{X}(M)$, 
and the curvature is given as
\belabel{coneR}
\begin{array}{rcl}
\xi \inter \hR&=&0,\\
\hR (X,Y)Z&=& \RR (X,Y)Z +  \gg(Y,Z)X - \gg(X,Z)Y .
\end{array}
\eeq
Hence, for the Ricci tensor we obtain that
\belabel{coneRic}
\begin{array}{rcl}
\xi \inter \hRic&=&0,\\[1mm]
\hRic (X,Y)&=& Ric (X,Y) +(n-1)  \gg(X,Y).
\end{array}
\eeq
This leads to the following observations:
\bs\label{coneprop}
Let $(\hm,\hg)$ be the cone over $(M,g)$.
\begin{enumerate}
\item $(M,g)$ has constant curvature $-1$   if and only if the cone $(\hm,
\hg)$ is flat.
 \item If $(\hm,\hg)$ is Einstein, then it is Ricci-flat.
\item If $(M,g)$ is Einstein with $Ric=(1-n)g$, then $(\hm,\hg)$ is Ricci-flat.
\end{enumerate}
\es

Finally we recall the important known fact that the existence of a time-like vector field $\xi$ with $\nabla\xi=\Id$ characterises cones locally, see for example \cite{GibbonsRychenkova98} or \cite[Lemma 1]{FedorovaMatveev14}. We include the proof here for expository reasons. 

\bs\label{locconeprop}
Let $(\hm,\hg)$ be a semi-Riemannian manifold of dimension $n+1$ that admits a time-like vector field $\xi$ such that $\widehat\nabla\xi=\Id$. Then there are local coordinates $(r,x^1, \ldots x^n)$ such that $\hg$ is of the form 
\[\hg=- \d r^2+r^2g_{ij}(x^1, \ldots , x^n) \d x^i\d x^j,
\]
where $i,j$ run from $1$ to $n$, we use the Einstein summation convention, and $g_{ij}=g_{ij}(x^1, \ldots, x^n)$ are functions of the $x^k$ coordinates only.
\es
\bprf
The vector field $\xi$ defines a positive function $r$ via
\[\hg(\xi,\xi)=-r^2.\]
Differentiating this relation gives
\[
2r\d r =\d(r^2)=-\d(\hg(\xi,\xi))=-2 g(\xi,\cdot )=-2\xi^\flat,\]
where the musical isomorphism $\flat$ denotes the metric dual with respect to $\hg$. Hence
\[\xi^\flat=-\d \left(\frac{r^2}{2}\right),
\]
is exact and therefore $\xi=-\widehat\nabla  \frac{r^2}{2}$ is a gradient vector field. The level sets of the function $r$ are orthogonal to $\xi$ and we can fix coordinates $( x^1, \ldots , x^n)$ on the level sets such that $(r, x^1, \ldots , x^n)$ are local coordinates on $\hm$. In these coordinates the metric has the form 
\[g=- \d r^2+\hg_{ij}(r, x^1, \ldots , x^n) \d x^i\d x^j,\]
and it holds $\xi=r\partial_r$.  Since $\hat\nabla \xi=\Id$, the vector field $\xi$ is a homothety,
\[\cal L_\xi \hg=2\hg,\] which implies that 
\[\hg_{ij}(r, x^1, \ldots , x^n)=r^2g_{ij}(x^1, \ldots , x^n) 
\]
for some functions $g_{ij}(x^1, \ldots , x^n)$ of the $x^i$ coordinates.
\eprf

\subsection{The holonomy of irreducible cones}
\label{irrsec}
For irreducible cones the possible holonomy groups are known from the Berger list \cite{berger55}, which comprises the orthogonal algebra and the three lists (\ref{berger1}--\ref{berger3}) below.
In the following let  $\h\subset \so(t+1,q)$  the irreducible holonomy algebra of a semi-Riemannian manifold $(\hm,\hg)$, i.e., one of the entries in Berger's list. For each possible $\h$ we will now determine if it can be the holonomy algebra of a cone.

\begin{enumerate}
\item $\h=\so(t+1,s)$: This is the holonomy algebra of a generic semi-Riemannian manifold of signature $(t+1,s)$.  
\begin{proposition}
Let  $(M,g)$ be a semi-Riemannian manifold of signature $(t,s)$ and of constant  curvature $\kappa \neq -1$ and let $(\widehat M,\hg)$ be the time-like cone
over $(M,g)$. Then $\hol(\widehat M,\hg)=\so(t+1,s)$.
\end{proposition}
\begin{proof}
The curvature endomorphisms of $(M,g)$ are of the form 
\[\RR (X,Y) =\kappa \big(  \gg(Y,\cdot )X - \gg(X,\cdot )Y \big).\]
Since the holonomy algebra contains all curvature endomorphisms, \eeqref{coneR} shows that
\[\so(t,s)\subset \hol(\hm,\hg),\]
where $\so(t,s)$ is embedded as the stabiliser of the vector $\xi$. Moreover, equations (\ref{coneLC}--\ref{coneR}) show that
\[
(\hnab_X\hR)(Y,Z)\xi=
-\hR (X,Y)Z 
=
-2\left(  \gg(Y,Z)X - \gg(X,Z)Y \right).
\]
This establishes $\hol(\widehat M,\hg)=\so(t+1,s)$.
\end{proof}

\item $\h$ is the holonomy of an irreducible symmetric space or  one of the following algebras: 
\begin{equation}\label{berger1}
\begin{array}{rcl}
\sp(1)\oplus\sp(p,q)&\subset& \so(2p,2q),\\
\sl(2,\rr) \oplus \sp(m,\rr)&\subset& \so(2m,2m),\\
\sl(2,\C)\oplus \sp(m,\C)&\subset& \so(4m,4m),
\end{array}\end{equation}
 where $p+q$ and $m$ are $>1$. In the first case the metric is quaternionic K\"ahler of signature $(4p,4q)$ and 
in the second it is quaternionic para-K\"ahler. 
Examples of the third type are obtained by complexifying manifolds with holonomy of the first two types, as discussed below.
In these examples $(\hm,\hg)$  is Einstein with {\em nonzero} Einstein constant, see \cite[Theorem 3]{AlekseevskyCortes05}.  
Hence, these  cases can be excluded  as holonomy of cones by 
\Cref{coneprop}.
\item $\h$ is one of the following:
\begin{equation}\label{berger2}
\begin{array}{rclrcl}
\mathfrak{u}(p,q),\ \su(p,q) &\subset & \so(2p,2q),&
 \sp(p,q)&\subset &\so(4p,4q),\\
\g_2&\subset &\so(7), &
\spin(7)&\subset & \so(8),\\
\g_{2(2)}& \subset & \so(3,4),&
\spin(3,4) &\subset&  \so(4,4)
.\end{array}\end{equation}
The geometric structures corresponding to these algebras do exist on cones over semi-Riemannian manifolds with certain structures. 
In fact, 
the following relations between structure on the base $(M,g)$ and on the cone are well known (see for example \cite{baer93} for the Riemannian case and \cite{kathhabil} for the indefinite cases, and references therein):
\begin{itemize}
\item[(i)]
The cone over a (semi-Riemannian) Sasaki, Einstein-Sasaki or  $3$-Sasaki manifold is, respectively, a K\"{a}hler, Ricci-flat 
K\"ahler or hyper-K\"{a}hler manifold  and hence has holonomy contained in $\mathfrak{u}(p,q)$, $\su(p,q)$ or $\mathfrak{sp}(p,q)$. 
\item[(ii)] The cone over a strict nearly-K\"ahler manifold of dimension $6$, 
Riemannian or of signature $(2,4)$, has a parallel 
$\mathbf{G}_2$- or $\mathbf{G}_{2(2)}$-structure and hence has holonomy contained in $\g_2$ or $\g_{2(2)}$.  
Similarly, the cone over a nearly para-K\"ahler manifold with $|\nabla J|^2\neq 0$ has holonomy contained in $\g_{2(2)}$, see \cite
[Prop.~3.1]{clss-09}.  
\item[(iii)]
The cone over a $7$-manifold with a nearly-parallel $\mathbf{G}_2$-structure, Riemannian or of signature $(3,4)$, has a parallel $\mathbf{Spin}(7)$- or $\mathbf{Spin}(3,4)$-structure and hence has holonomy contained in $\spin(7)$ or $\spin(3,4)$.  
\end{itemize}
The question remains, whether the holonomy of the cone is not only contained but actually {\em equal} to one of the algebras in the list (\ref{berger2}). In the Riemannian setting (which corresponds to the case where the base of the time-like cone is negative definite) this can be established by using  Gallot's Theorem  that the (space-like) cone over a {\em complete} Riemannian manifold $(M,g)$  is either flat or irreducible and then by constructing a complete $(M,g)$ with the corresponding structure. 
For indefinite metrics several gaps open up in this argument: our generalisation of Gallot's Theorem in \cite{acgl07} assumes that $(M,g)$ to be compact and complete and implies that the cones is flat or {\em indecomposable}, but not necessarily irreducible. 
Hence, even if one constructed compact and complete indefinite semi-Riemannian manifolds with the above structures, the cone would not have to be irreducible and hence its holonomy could be an indecomposable, non irreducible subalgebra of the algebras in (\ref{berger2}).
We suspect however, that for a ``generic'' semi-Riemannian manifold with one of the above structures, the cone has holonomy equal to the algebras in (\ref{berger2}).  
An explicit way of constructing examples of cones with special holonomy is given below in Remark \ref{realformremark}.

\item
$\h$ is one of the following algebras:
\begin{equation}\label{berger3}
\begin{array}{rclrcl}
\so(n,\C)&\subset&\so(n,n)& \sl (2,\C)\oplus \sp(m,\C)&\subset& \so(4m,4m)
\\
\g_2^\C&\subset& \so(7,7),&\spin(7,\C)&\subset&\so(8,8).
\end{array}
\end{equation}
Examples can be obtained by complexification as we will explain now in detail. In the case of  $\sl (2,\C)\oplus \sp(m,\C)$ the metric is then Einstein of nonzero scalar curvature (incompatible with a cone), whereas in the  two exceptional cases
it is Ricci-flat.
\end{enumerate}
\subsubsection*{Realisation of complex holonomy algebras}
Let $(M,g)$ be a connected real analytic manifold endowed with a real analytic semi-Riemannian metric. Then it is easy to see that 
$M$ can be embedded into a connected complex manifold $M^\mathbb{C}$ with the following properties. 
\begin{enumerate}
\item 
There exists an
atlas of $M^\mathbb{C}$ such that each of its charts $\varphi : U \rightarrow \mathbb{C}^n$ is real-valued on 
$U\cap M$ and the restrictions $\varphi|_{U\cap M} : U \cap M \rightarrow \mathbb{R}^n$, $U\cap M\neq \emptyset$, form an atlas of $M$. 
\item The metric coefficients $g_{ij}(x)$ with respect to the real coordinates $x=(x^1,\ldots , x^n)= \varphi|_{U\cap M}$ 
are given by real power series converging in $U\cap M$. 
\item The power series $g_{ij}(z)$ in the holomorphic coordinates
$z=(z^1,\ldots , z^n)=\varphi$ converges in $U$  for all $i,j$. 
\end{enumerate}
It follows that we can define a holomorphic symmetric tensor field $g^\mathbb{C}$ on $M^\mathbb{C}$  by 
\[ g^\mathbb{C}|_U = \sum g_{ij}(z)\d z^i\d z^j.\]
The tensor field is non-degenerate on a neighborhood of $M$ and by restriction we can always assume that it is non-degenerate on 
$M^\mathbb{C}$. Then it defines what is called a \emph{holomorphic Riemannian metric} on $M^\mathbb{C}$. 
We will call $(M^\mathbb{C}, g^\mathbb{C})$ a \emph{complexification} of $(M,g)$. 
Recall that a pair consisting of a complex manifold and a holomorphic Riemannian metric on that manifold is 
called a \emph{holomorphic Riemannian manifold}. Note that $(M^\mathbb{C}, g^\mathbb{C})$ is unique as a germ of 
holomorphic Riemannian manifold along $M$. 

We define the \emph{holonomy algebra} of a holomorphic Riemannian manifold $(M^\mathbb{C}, g^\mathbb{C})$ at $p\in M^\mathbb{C}$ 
as the Lie algebra spanned by all the skew-symmetric endomorphisms 
\[ ((\nabla^\mathbb{C})^k_{v_1,\ldots ,v_k} R^\mathbb{C})(v_{k+1}, v_{k+2}) \in \mathfrak{so}(T^{1,0}_pM^\mathbb{C})\cong \mathfrak{so}(T_pM)^\mathbb{C},\] 
where $v_1, 
\ldots , v_{k+2}\in T^{1,0}_pM^\mathbb{C}$ 
and $k\ge 0$. Here $\nabla^{\mathbb{C}}$ denotes the (holomorphic) Levi-Civita connection of $g^\mathbb{C}$ and $R^\mathbb{C}$ its curvature tensor. 

\begin{proposition}\label{prop24} Let $(M^\mathbb{C}, g^\mathbb{C})$  be a complexification of a connected semi-Riemannian
manifold $(M,g)$.  Then the holonomy algebra of $(M^\mathbb{C}, g^\mathbb{C})$
is given by the complexification $\mathfrak{h}^\mathbb{C}$ of the holonomy algebra $\mathfrak{h}$ of $(M,g)$. 
\end{proposition}

\begin{proof} By the Ambrose-Singer theorem for real analytic semi-Riemannian manifolds we know that $\mathfrak{h}$
is spanned by all the endomorphisms $(\nabla^k_{v_1,\ldots ,v_k} R)(v_{k+1}, v_{k+2})\in \mathfrak{so}(T_pM)$, where $v_1, 
\ldots , v_{k+2}\in T_pM$ 
and $k\ge 0$.  From the 
definition of $g^\mathbb{C}$ as complex-analytic extension of $g$ it is clear that the Levi-Civita connection $\nabla^\mathbb{C}$ 
of $g^\mathbb{C}$ coincides with the complex-analytic extension of the Levi-Civita connection $\nabla$ of $g$. The same
relation holds for the curvature tensors and their covariant derivatives. This implies the proposition. 
\end{proof}
Next we consider the real analytic manifold $N$ of dimension $2n$ underlying the complex manifold 
$M^\mathbb{C}$. It carries a corresponding integrable complex structure  $J$ and we can identify $(N,J)$ with $M^\mathbb{C}$. We endow $N$ with the real analytic semi-Riemannian metric 
\begin{equation} \label{gN:eq}  g_N := 2\, \mathrm{Re}\, g^\mathbb{C}.\end{equation} 
Note that $g_N$ can be considered as a (fibrewise) real bilinear form on $TN$ by means of the canonical identification
\[ TN \cong T^{1,0}N, \quad X\mapsto X^{1,0}=\frac12 (X-iJX).\] 
The factor $2$ in \re{gN:eq}  is chosen such that $g^\mathbb{C}$ is obtained by restricting (the complex bilinear extension of) $g_N$ 
to $T^{1,0}N$. 

We observe that the metric $g_N$ can be defined on the real analytic manifold $N$ underlying any 
holomorphic Riemannian manifold $(M^\mathbb{C}, g^\mathbb{C})$ irrespective of whether 
$(M^\mathbb{C}, g^\mathbb{C})$ is a complexification of a semi-Riemannian manifold $(M,g)$. 
\begin{theorem} 
\label{choltheo} Let $(M^\mathbb{C}, g^\mathbb{C})$ be a connected holomorphic Riemannian manifold
and $(N,g_N)$ the corresponding semi-Riemannian manifold. Then  $(N,g_N)$ has neutral signature and 
its holonomy algebra is isomorphic to the holonomy algebra of $(M^\mathbb{C}, g^\mathbb{C})$. 
\end{theorem}

\begin{proof} 
Note first that $g_N(J\cdot , J\cdot ) = -g_N$, since $g_N$ is of type $(2,0) + (0,2)$ with respect to $J$. 
This implies that $g_N$ has neutral signature, since $J$ it maps a maximal definite subspace of $T_pN$ to a 
maximal definite subspace of the same dimension and of opposite signature.

We consider first the Lie algebra $\frak{so}(T_pN)$, $p\in N$, with respect to $g_N$ and its subalgebra
\[ \frak{so}(T_pN)^J := \{ A \in \frak{so}(T_pN) \mid [A, J]=0\}.\]
The latter can be considered as a complex Lie algebra with the complex structure $A \mapsto JA$. Indeed, 
$JA$ is $g_N$-skew-symmetric as the product of a symmetric with a commuting skew-symmetric endomorphism. 
The symmetry of $J$ follows from the fact that  $J$ is an anti-isometry squaring to minus one. 

We claim that $\frak{so}(T_pN)^J$ 
is canonically isomorphic to the complex Lie algebra $\frak{so}(T^{1,0}_pN)$ with respect to $g^\mathbb{C}$. 
Using the metric $g_N$, we can identify $\frak{so}(T_pN)^J$ with the set of real points in $\bigwedge^{2,0}T_pN \oplus \bigwedge^{0,2}T_pN$
and the latter can be identified with $\bigwedge^{2,0}T_pN \cong \bigwedge^2T^{1,0}M$ by projecting to the $(2,0)$-component. Finally, using the metric $g^\mathbb{C}$,
we can identify $\bigwedge^2T^{1,0}M$ with $\frak{so}(T^{1,0}_pN)$. This yields a canonical isomorphism 
\begin{equation} \label{iso:eq}  \Phi: \frak{so}(T_pN)^J \rightarrow \frak{so}(T^{1,0}_pN)\end{equation} 
of complex vector spaces. It simply maps $A\in \frak{so}(T_pN)^J$ to its restriction to $T^{1,0}N$. Therefore it is even an isomorphism 
of Lie algebras.  

Next we show, for all $v_1,\ldots v_{k+2}\in T_pN$, that  under the canonical isomorphism \re{iso:eq} the tensor $(\nabla^N)^k_{v_1,\ldots ,v_k}R^N(v_{k+1},v_{k+2})$ is mapped 
to $(\nabla^\mathbb{C})^k_{w_1,\ldots ,w_k} R^\mathbb{C}(w_{k+1},w_{k+2})$, where $w_j= v_j^{1,0}$, $\nabla^N$ denotes the Levi-Civita connection 
of $g_N$ and $R^N$ its curvature. This implies the theorem, in virtue of the Ambrose-Singer 
theorem. First we show that $\nabla^N$ can be constructed from the holomorphic connection $\nabla^\mathbb{C}$. 
Let $\nabla'$ be the unique connection in $(TN)^\mathbb{C}$ with the following properties: 
\begin{enumerate} 
\item $\nabla'_ZW=\nabla^\mathbb{C}_ZW$ for all holomorphic vector fields $Z, W$ on $M^\mathbb{C}$. 
\item $\nabla'_{\bar{Z}}W=0$ for all holomorphic vector fields $Z, W$ on $M^\mathbb{C}$. 
\item $\nabla'$ is real, that is restricts to a connection in $TN$. 
\end{enumerate}
Notice that the above properties imply that the subbundles $T^{1,0}N$ and $T^{0,1}N$ 
are $\nabla'$-parallel and, hence, that $\nabla'J=0$. 
Moreover, using these properties, it is straightforward to check that $\nabla'$ is metric torsion-free, since  $\nabla^\mathbb{C}$ is. 
This implies that $\nabla'$ (when considered as a connection in $TN$) coincides with the Levi-Civita connection 
$\nabla^N$. As a consequence, we see that $\nabla^NJ=0$ and thus $(\nabla^N)^k_{v_1,\ldots ,v_k}R^N(v_{k+1},v_{k+2})\in \frak{so}(T_pN)^J$. 
Now let $X, Y$ be real vector fields on an open set $U\subset N$ which are infinitesimal 
automorphisms of $J$. Then we have the formula 
\begin{equation}  \label{fund:eq} (\nabla_X^NY)^{1,0} = \nabla^\mathbb{C}_{X^{1,0}}Y^{1,0}.\end{equation}
This follows immediately from the defining properties of $\nabla'=\nabla^N$ by decomposing
$X = Z + \bar{Z}$ and $Y=W + \bar{W}$, where $Z=X^{1,0}, W=Y^{1,0}$ are holomorphic. 
From \re{fund:eq} we deduce that 
\[ \left( \left( (\nabla^N)^k_{v_1,\ldots ,v_k}R^N(v_{k+1},v_{k+2})\right) v_{k+3}\right)^{1,0} = \left((\nabla^\mathbb{C})^k_{w_1,\ldots ,w_k} R^\mathbb{C}(w_{k+1},w_{k+2})\right)w_{k+3},\] 
for all $v_1,\ldots ,v_{k+3}\in T_pN$, where we recall that $w_j=v_j^{1,0}$. 
Since the left-hand side is precisely 
\[ \Phi  \left( (\nabla^N)^k_{v_1,\ldots ,v_k}R^N(v_{k+1},v_{k+2})\right) w_{k+3},\]
we can conclude that 
\[ \Phi  \left( (\nabla^N)^k_{v_1,\ldots ,v_k}R^N(v_{k+1},v_{k+2})\right)= (\nabla^\mathbb{C})^k_{w_1,\ldots ,w_k} R^\mathbb{C}(w_{k+1},w_{k+2}),\]
finishing the proof. 
\end{proof}
This leads to the following consequence:
\bfolg\label{cor26}
The complex holonomies 
\begin{equation}\label{berger4}
\so(n,\C)\subset\so(n,n), \quad
\g_2^\C\subset \so(7,7),\quad \spin(7,\C)\subset\so(8,8),
\end{equation}
are holonomy algebras of time-like cones.
\efolg
\bprf 
This follows from the above considerations and from the fact that the compact real forms of the complex holonomy algebras in (\ref{berger4}) can be realised by timelike cones over negative definite manifolds. Indeed, if $(\hm,\hg)$ is a time-like cone with  holonomy $\so(n)$, $\g_{2} \subset  \so(7)$, or $
\spin(7) \subset  \so(8)$, then there is the Euler vector field $\xi \in \Gamma(T\hm)$. Hence the real analytic metric $\hg^\C$ on $\hm^\C$ has the   holomorphic Euler vector field $\xi^\C$ with  $\hnab^\C\xi^\C=\Id$. On the real manifold $N=\hm^\C$  we then have that $\eta=2 \mathrm{Re}\, \xi^\C$ satisfies $\nabla^{N}\eta=\Id$, as a consequence of equation (\ref{fund:eq})
applied here to $N=\hm^\C$ instead of $M^\C$. By  \Cref{locconeprop} we then get that $N$ is locally a cone, which by \Cref{choltheo} has one of the complex holonomies in (\ref{berger4}) as holonomy algebra.
\eprf
This proof can be made more explicit in local coordinates. 
Locally the metric $\hg$ is of the form \[\hg=-\d r^2+r^2 g_{ij}(x^k)\d x^i \d x^j\] with Euler vector field $\xi=r\partial_r \in \Gamma(T\hm)$. The analytic  metric  $\hg^\C$ on $\hm^\C$ then is of the form \[\hg^\C=-\d u^2 +u^2 g_{ij}(z^k) \d z^i \d z^j\] with coordinates $(u=r+is,z^1, \ldots, z^n)$ with $z^k=x^k+\mathrm{i} y^k$ and  holomorphic Euler vector field $\xi^\C=u\partial_u$ with  $\hnab^\C\xi^\C=\Id$. 
Then the metric $\widehat{h}=\frac12 g_N$ on $N=\widehat{M}^\C$ is given by 
\[ \widehat{h} = -\d r^2 +\d s^2+(r^2-s^2) \mathrm{Re}(g_{ij} (z^k) \d z^i \d z^j) - 2 r s\, \mathrm{Im} (g_{ij} (z^k) \d z^i \d z^j).\]
One can directly check that $\eta=r\partial_r+s\partial_s$ satisfies $\nabla^N\eta=\Id$. Moreover, the cone coordinate with respect to $\widehat{h}$ is given by $\rho=\sqrt{r^2-s^2}$, which satisfies $\widehat{h}(\eta , \cdot ) =-\rho\d \rho$.

\bbem\label{realformremark}
Finally, we note that it is possible to construct examples of pseudo-Riemannian cones with these holonomies using different real forms of the  complexified metrics and Proposition \ref{prop24} and  Corollary \ref{cor26}: For example $\SL(2,\rr)\times \SL(2,\rr)$ admits a unique left-invariant nearly pseudo-K\"{a}hler structure, which is a different real form of the complexification of the Riemannian nearly K\"{a}hler structure on $\SU(2)\times \SU(2)$, \cite{SchaferSchulte-Hengesb10}. Hence the cone metrics are different real forms of the holomorphic cone metric. Since the cone over $\SU(2)\times \SU(2)$ has holonomy $\g_2$, the cone over $\SL(2,\rr)\times \SL(2,\rr)$ must have holonomy equal to $\g_{2(2)}$.
\ebem

\subsection{Manifolds with parallel  null line bundle}
In the following manifolds with a parallel null line bundle will be crucial. In this section we will collect some facts about them.

Let $(M,g)$ be a semi-Riemannian manifold with a {\em parallel null line bundle $\L$}, i.e., $\L$ is a rank $1$ subbundle of $TM$ the  fibres of which are null  with respect to the metric $g$ and invariant under parallel transport with respect to the Levi-Civita connection $\nabla$ of $g$. 
This implies that every non-vanishing section $\chi\in \Gamma(\L)$ satisfies 
\begin{equation}\label{recur}
\nabla \chi =\alpha\otimes \chi,\end{equation}
for a uniquely determined $1$-form $\alpha$.
Any  vector field  that satisfies  \eeqref{recur} for some $1$-form $\alpha$ is called a {\em recurrent vector field}. 
\begin{proposition}
Let $\chi$ be a recurrent vector field on a connected semi-Riemannian manifold
$(M,g)$. Then the function $f=g(\chi,\chi) $ is either  everywhere   positive,  negative or  zero. In particular, $\chi $ can only have zeros if $f\equiv 0$. 
\end{proposition}
\begin{proof}
The equation (\ref{recur}) yields the ODE
$X(f)=2 \alpha(X)f$  for every vector field $X$.  These ODEs imply that if $f$ vanishes at a point, then $f$ vanishes in a neighbourhood of this point. 
Due to the continuity of $f$ this shows that 
$M$ is a disjoint union of the three open sets $\{f>0\}$, $\{f<0\}$ and $\{f=0\}$. Now, since $M$ is connected, the proposition follows. 
\end{proof}
Hence,
 locally the existence of a parallel null line bundle  is equivalent to the existence of a {\em recurrent null vector field}, where we recall that a vector field $\chi$ is null if $g(\chi,\chi)=0$ and $\chi$ does not vanish \cite[Definition 3 in Chapter 3]{oneill83}. 
 Moreover, a nowhere vanishing
recurrent vector field $\chi$  can be rescaled to parallel vector field $ \lambda \cdot \chi$, for a  non-vanishing function $\lambda$, if and only if  the $1$-form $\alpha $ is exact. Indeed, if  $\alpha=dh$, then 
\[  \nabla (\mathrm{e}^{-h} \chi)=0,\]
then $\lambda=\mathrm{e}^{-h}$ so that $\lambda\cdot \chi$ is parallel\footnote{ 
For non-null  recurrent vector fields this shows that they  can always   be  rescaled locally to a parallel vector field, as  
$0=\R(X,Y,\chi,\chi)= \d\alpha(X,Y)\, g(\chi,\chi)$ yields that  $\alpha$ is  closed, or more explicitly, $\alpha =\frac{1}{2} \d\ln f $ with $f= g(\chi,\chi)$ and hence $\frac{1}{\sqrt{f}}\chi$ is parallel.}.
Conversely, if  $\lambda \cdot \chi$ is parallel, then 
\[
0=\R (X,Y)\chi=\d\alpha(X,Y)\chi\]
for all $X,Y\in TM$.

Hence, 
on  simply connected manifolds  $(M,g)$,   nowhere vanishing recurrent vector fields can be rescaled to parallel ones if and only if $\alpha $ is closed.
The choice we have when locally choosing a recurrent vector field that spans a null line bundle $\L$ can be used to find special recurrent sections of $\L$.
\blem\label{recurlemma}
Let $\L$ be a parallel null line bundle. Then locally there is a recurrent gradient vector field $\chi$ which spans $\L$. This vector field satisfies that $\nabla\chi= h\, \chi^\flat \otimes\chi$ for a function $h$.
\elem
\bprf
Since $\L$ is parallel, the hyperplane distribution $\L^\perp=\{X\in TM\mid g(X,\L)=0$  is parallel and hence involutive. Hence by Frobenius' Theorem $\L^\perp$ is integrable and the integral manifolds are given as 
$f\equiv constant$ for some local function $f$. Hence $\L^\perp=\ker(df)$ and the gradient $\chi:=\grad( f)$ of $f$ spans $\L$. Then $\chi$ is recurrent, i.e.,  $\nabla \chi=\alpha\otimes \chi$. But then $\chi=\grad( f)$ implies  that 
\[0=\d\chi^\flat = \alpha\wedge \chi^\flat,\]
which shows that $\alpha=h\,\chi^\flat $ for a local function $h$.
\eprf
\section{Cones with parallel null lines}
\label{localsec}

In this section we assume that the cone (\ref{conemetric}) over a semi-Riemannian manifold $(M,g)$ admits a null line that is invariant under parallel transport. We will show that locally this implies that the cone admits a parallel null vector field and that the base $(M,g)$ 
is locally an exponential extension of a semi-Riemannian manifold $(M_0,g_0)$, see Definition \ref{ext:def}. The total space of the cone will then be shown to be locally isometric to a double warped extension $(\widetilde{M},\tg)$ of $(M_0,g_0)$, see Definition \ref{ext:def}. This will generalise our results  for Lorentzian cones in 
\cite[Section 9]{acgl07}.
\bs\label{parallel:prop}
Let $(\hm,\hg)$ be a timelike cone and assume that   $(\hm,\hg)$ admits a parallel null line $\mathbf{L}$. Then
the following holds: 
\begin{enumerate}
\item[(i)] The set $\hm_{\mathrm{reg}}$ where $\mathbf{L}$ is not perpendicular to the Euler vector field $\xi$ is open and dense
and invariant under the flow of $\xi$. So, in particular, $\hm_{\mathrm{reg}} = \widehat{M_{\mathrm{reg}}}$, where 
$M_\mathrm{reg} := \hm_{\mathrm{reg}} \cap M$. 
\item[(ii)] $\mathbf{L}$ is flat and, hence,  locally (and globally if $M$ is simply connected) there is a parallel 
null vector field that spans $\mathbf{L}$. 
\end{enumerate}
\es
\bprf 
By passing to the universal cover of $(\hm, \hg)$, that is to the cone over the universal cover of $M$, we can assume that $M$ and $\hm$ are simply connected. Hence, we can assume that the parallel null line $\mathbf{L}$ is spanned by
 a nowhere vanishing recurrent vector field 
 $\chi$ on $(\hm,\hg)$. Then we decompose 
\[\chi=f\xi+Z,\]
where $Z$ is tangent to $\M$ and nowhere vanishing. 
We claim that the function $f$ cannot vanish on a nonempty open set. If it did, formulae \re{coneLC}
 would give
\[ 
\alpha(X)Z=\hnab_X   \chi=\nabla_XZ +\gg (X,Z)\xi,
\]
on the open set, and hence 
 $\gg (X,Z)=0$ for all $X\in T\M$, which is a contradiction. 
This proves that the open set $\hm_{\mathrm{reg}}=\{p\in \hm\mid f(p)\not=0\}$ is dense. The invariance of $\hm_{\mathrm{reg}}$ under the homothetic flow of $\xi$ follows from the invariance of $\mathbf{L}$ under the flow. The latter 
is obtained by writing the Lie derivative as
$\mathcal{L}_\xi = \widehat{\nabla}_\xi -\mathrm{Id}$ and using  that $\mathbf{L}$ 
is parallel.

On $\hm_{\mathrm{reg}}$ we have 
\[
\d\alpha(X,\xi)\chi=\hR(X,\xi)\chi=0\]
and 
\[\d\alpha(X,Y)\chi= \hR(X,Y)\chi= \hR(X,Y)Z\in  T\M,\]
for $X, Y\in TM$. This implies $\d\alpha=0$, since $\hm_{\mathrm{reg}}$ is dense, proving that $\mathbf{L}$ is flat. 
\eprf
In the next proposition we describe an example of a cone with a parallel null line before showing that 
every  example is locally of this form. 


\begin{definition} \label{ext:def} Let $( M_0, \gg_0)$ be a semi-Riemannian manifold. Then the warped products $(\M=\rr\times M_0,\gg=ds^2+\e^{2s}\gg_0)$ and $(\widetilde{M} =\mathbb{R}^+ \times \mathbb{R}^- \times M_0, \widetilde{g}=2\, \d u\, \d v+u^2\gg_0)$ will be called the \emph{exponential extension} and  the \emph{double warped extension} of $(M_0, \gg_0)$, respectively. 
\end{definition}
\bs\label{warpconelemma} Let $( M_0, \gg_0)$ be a semi-Riemannian manifold. 
The time-like cone $(\hm,\hg)$ over the exponential extension $(M,g)$ of $( M_0, \gg_0)$  is globally isometric to 
the double warped extension  $(\widetilde{M} , \widetilde{g})$ of $(M_0, \gg_0)$. In particular, the cone admits the parallel null vector field 
$\del_v$.
\es
\bprf
The cone metric over $(\M,\gg)$ is given by
\[ \hg=- \d r^2+r^2 \d s^2 +r^2\mathrm{e}^{2s}\gg_0,\]
with $r\in \rr^+$ and $s\in \rr$.
For the diffeomorphism
\belabel{psi:eq} \psi: \hm=\rr^+\times \rr\times M_0\ni (r,s,p)\longmapsto \left(u=r\e^{s}, v=-\tfrac{1}{{2}}r\e^{-s}, p\right)\in \widetilde{M}=\rr^+\times \rr^-\times  M_0\eeq
one checks that
\[
(\psi^{-1})^*\hg = 2\,\d u\,\d v+u^2\gg_0.\]
This proves the statement.
\eprf
\btheo \label{firstmain:thm}
Let $(\hm,\hg)$ be a time-like cone over a semi-Riemannian manifold $(M,g)$. Assume that 
$(\hm,\hg)$ admits a parallel null line $\mathbf{L}$.  
Then the open dense subset $\hm_{\mathrm{reg}} \subset (\hm, \hg )$, cf.\ Proposition \ref{parallel:prop}, is locally isometric to the double warped extension $(\widetilde{M} , \widetilde{g})$ of  a semi-Riemannian manifold $(M_0, \gg_0)$ and the open dense subset $M_{\mathrm{reg}}\subset (M,g)$ is locally isometric to the exponential extension of $( M_0, \gg_0)$.\etheo
\bprf
Since we have to show the existence of a {\em local} isometry,
by Proposition \ref{parallel:prop}, we can assume that $\mathbf{L}$ admits a parallel trivializing section $\chi$. 
We write the parallel null vector field $\chi$ on $\hm$ as
\[\chi=\hat f\xi +\hat Z\]
with $\hat Z$ a nowhere vanishing vector field tangent to $\M$  and $\hat f$ a smooth function on $\hm$. We will show that $\hat Z$ defines  vector field $Z$ on $\M$. From
\[ [ \xi, \hat Z] =
- d\hat f(\xi)\xi +[\xi, \chi]
=
- d\hat f(\xi) \xi - \hnab_\chi\xi 
=
- d\hat f(\xi) \xi - \chi =
 -(d\hat f (\xi) + \hat f) \xi -\hat Z, \]
with $[ \xi, \hat Z] $ being tangent to $\M$, 
we get on the one hand that
\[\d\hat f(\xi) +\hat f = r\del_r (\hat f)+\hat f=0,\]
and on the other that
\[[ \xi, \hat Z]+\hat Z =0\]
The first equation shows that 
\[\hat f= \frac{1}{r} f \]
with $f$ a function on $\M$ and the second that
\[ \hat Z=\frac{1}{r} Z,\]
with  $Z= r\hat Z$  a vector field on $\M$, i.e.,  $[\xi,Z]=0$.
Hence we have
\[\chi=\frac{1}{r} (f\xi +Z). \]
Differentiating in  direction of $X\in T\M$, by \re{coneLC} we get
\[
0
=
r\hnab_X \chi=
(df(X)+\gg(X,Z))\xi+f X+\nabla_XZ.
\]
which shows that
\[ Z=-\mathrm{grad}(f), 
\]
where $\mathrm{grad}$ denotes the gradient with respect to $\gg$, and
\belabel{nabZ10}
\nabla Z= -f\,\Id.
\eeq
Hence, the distribution $Z^\perp$ on $\M$ is integrable and its leafs are given by the level sets of $f$. The vector field  $Z$ is not only a gradient but also a conformal vector field, since from \re{nabZ10} we compute
\[\cal L_Z\gg=-2 f\gg. \]
Note also that on $M_{\mathrm{reg}} = \widehat{M}_{\mathrm{reg}}\cap M$, the vector field $Z$ is transversal to the level sets of $f$.
This follows from $df(Z)=g(\mathrm{grad}(f),Z)=-g(Z,Z)=-f^2$. 
Hence, locally on $M_{\mathrm{reg}}$ the metric $\gg$ is given as
\[
\gg=\frac{(\d f)^2}{f^2} +f^2\gg_0
\]
where $c^2g_0$ is the metric $\gg$ restricted to a level set $\{ f=c\}$. Setting $s=\log |f|$ and using  \Cref{warpconelemma} this proves the statement in the Theorem.
\eprf

The local geometry described in this theorem is summarised in diagram (\ref{diagram}) in the introduction. 
We also have the following global result. 
\btheo \label{glob:thm}
Let $(\hm,\hg)$ be a time-like cone over a simply connected and space-like complete semi-Riemannian manifold $(M,g)$. Assume that 
$(\hm,\hg)$ admits a parallel null line $\mathbf{L}$ which is nowhere perpendicular to $\xi$.  
Then $(\hm, \hg )$ is isometric to the double warped extension $(\widetilde{M} , \widetilde{g})$ of  a semi-Riemannian manifold $(M_0, \gg_0)$ and $(M,g)$ is  isometric to the exponential extension of $( M_0, \gg_0)$, cf.\ Definition \ref{ext:def}.
\etheo
\bprf
Since $M$ (and thus $\hm$) is simply connected, the flat line bundle $\mathbf{L}$ (see Proposition~\ref{parallel:prop}) admits a 
parallel section $\chi\neq 0$. By assumption, the function $\hg(\chi,\xi)$ has no zeroes.  As in the proof of Theorem \ref{firstmain:thm}, 
we can thus write 
\[ \chi=\frac{1}{r} (f\xi +Z)\]
for a nowhere vanishing function $f$ and $Z= -\mathrm{grad}(f)$ on $M=M_{\mathrm{reg}}$. 
Then $Z'=\frac{1}{f}Z$ is a space-like geodesic unit vector field, as follows from $g(Z,Z)=f^2$ and the equation (\ref{nabZ10}):
\[ \nabla_ZZ'= -\frac{df(Z)}{f^2}Z -Z = 0.\]
From the space-like completeness assumption  we conclude that $Z'$ is complete, giving rise to a global
diffeomorphism $M \simeq\mathbb{R}\times M_0$ under which the metric takes the form $g=ds^2 +e^{2s}g_0$. 
\eprf

\bbem The assumption that $\mathbf{L}$ is nowhere perpendicular to $\xi$ in Theorem \ref{glob:thm} cannot be dropped. 
In fact, the universal covering of anti-de Sitter space is simply connected and complete 
but any parallel line distribution over the time-like cone is somewhere perpendicular 
to $\xi$. In fact, a complete Lorentzian metric of constant negative curvature cannot be globally written in the form $ds^2 +e^{2s}g_0$, 
since the latter metric is incomplete, see \cite[Proposition 2.5]{acgl07}. Locally it admits such description, where the Lorentzian metric $g_0$ is 
moreover flat.  
\ebem

In  the following we will study metrics of the form $\tg=2 \d u\d v+u^2g_0$. For brevity we will  drop the index $0$ at $g_0$.

\section{Metrics of the form $\tg=2\d u\,\d v +u^2 \gg$}
\label{uvgsec}

\subsection{Levi-Civita connection, curvature and holonomy}
Let $\gg$ be a semi-Riemannian metric (of signature $(t,s)$) on a manifold $\M$ of dimension $n$.  We want to study the geometry and the holonomy of  metrics  of signature $(t+1,s+1)$ of the form
\belabel{uvcone}  \tg =2\d u\,\d v +u^2 \gg,\end{equation}
from now on to be considered on the maximal domain $\widetilde{M}:=\rr^+\times  \rr\times \M \supset \rr^+\times  \rr^-\times \M$.  
Such metrics admit a $2$-dimensional solvable group of homotheties given by
$(u,v,p)\mapsto (\e^{r}u,\e^{r}v+s,p)$.
Its infinitesimal generators are the parallel null vector field $\del_v$ and the homothetic vector field $U=u\del_u+v\del_v$.

There are obvious inclusions of $\M=\{1\}\times \{0\}\times \M\subset \tem$, $T\M\subset T\tem $ and $\Gamma(T\M)\subset \Gamma (T\tem)$. Using these identifications, the Levi-Civita connection of $\tg$ can be expressed by
\belabel{uvconenabla}
\begin{array}{rcl}
\tnab_XY = \nabla_X Y-u\, \gg(X,Y)\, \del_v  &\text{ and }&
\tnab _X\partial_u = \dfrac{1}{u}\ X
\end{array}
\end{equation}
with $X\in TM$, $Y\in \Gamma (T\M)$, $\nabla$ the Levi-Civita connection of $\gg$, and all other derivatives either vanish or are determined by the vanishing of the torsion of $\nabla$.  Note that the homothetic vector field $U=u\del_u+v\del_v$ satisfies $\tnab U=\mathrm{Id}$.
Moreover, for the curvature of $\tg$ one computes that
\belabel{uvconeR}
\begin{array}{rcl}
\del_v\inter \widetilde{\RR}&=& \del_u\inter \widetilde{\RR}\ =\ 0,
\\
\widetilde \R(X,Y)Z&=& \R(X,Y)Z, \quad \text{ for all }X,Y,Z\in T\M,\end{array}\end{equation}
where $\R$ is the curvature tensor of $(\M,\gg)$. Note that this implies for an arbitrary tensor field $Q$ that
\belabel{commute}
\tnab_{\del_u}\tnab_X Q=
\tnab_X\tnab_{\del_u} Q
\end{equation}
for all $X\in \Gamma (T\M)$.

For the derivatives of $\widetilde{\RR}$ we get the following formulae, which determine all possible  derivatives. 
First we observe that
\belabel{deluuR}
(\tnab_{\del_u}\widetilde{\RR})(\del_u,X)=0
\end{equation}
for all $X\in T\M$.
For the $q$-th derivative in $\del_u$-direction we compute
\belabel{delupR}
(\tnab_{\del_u} \cdots \tnab_{\del_u}\widetilde{\RR})(X,Y)Z=
\frac{(-1)^q(q+1)!}{u^q}\,\R(X,Y)Z.
\end{equation}
Moreover, a simple induction shows
\belabel{xpR}
\begin{array}{rcl}
(\tnab_{X_1} \cdots \tnab_{X_p}\widetilde{\RR})(X,Y)Z
&=&(\nabla_{X_1} \cdots \nabla_{X_p}\R)(X,Y)Z
\\
&&{}-u
\sum_{i=1}^p

(\nabla_{X_1}\stackrel{i\atop \uparrow}{\cdots }
\nabla_{X_p}\R)(X,Y,Z,X_i)\del_v,
\end{array}
\end{equation}
for all $X_i,X,Y,Z,W\in T\M$ and where the symbol $\stackrel{i\atop \uparrow}{\ldots}$ indicates the omission of the $i$th term.
In general, a straightforward computations shows
\begin{proposition}\label{curvprop}
The $(p+q)$th derivative of $\widetilde{\RR}$ is determined  by the relations  
\begin{eqnarray*}
\del_v\inter \tnab^k\widetilde{\RR}&=&0,\\ 
(\tnab_{\del_u}\tnab_{X_1}\cdots 
\tnab_{X_p}\widetilde{\RR})(Y,Z)(W)
&=&
(\tnab_{X_1}\tnab_{\del_u}\tnab_{X_2}\cdots 
\tnab_{X_p}\widetilde{\RR})(Y,Z)(W),
\end{eqnarray*}
and the formula
\be
\lefteqn{
(\tnab_{\del_u} \cdots \tnab_{\del_u}\tnab_{X_1}\cdots 
\tnab_{X_p}\widetilde{\RR})(X,Y)Z=}
\\
&=&
\frac{c(p,q)}{u^q}\left((
\nabla_{X_1}\cdots 
\nabla_{X_p}\R)(X,Y)Z
-u \sum_{i=1}^p
(\nabla_{X_1}\stackrel{i\atop \uparrow}{\cdots }
\nabla_{X_p}\R)(X,Y,Z,X_i)\del_v\right),
\ee
where $c(p,0)=1$ and $c(p,q)=(-1)^q(p+2)\cdot \ldots \cdot (p+q+1)$ when $q\ge 1$.
\end{proposition}

Our aim is to study the holonomy of metrics $\tg=2 \d u\d v+u^2g$. Since $\del_v$ is a parallel vector field on $(\tem,\tg)$, the holonomy of $(\tem,\tg)$ is contained in the stabiliser of the vector $\del_v$ at a point. By splitting  $T\tem=\rr\del_v\+ TM\+ \rr\del_u$, where
$\mathrm{span}\{ \del_v, \del_u \} = TM^\perp$, and fixing an orthonormal basis in $T_p M$ we can identify  $\so(T_pM,g)\simeq\so(t,s)$ and have $\hol(M,g)\subset \so(t,s)$. Hence, we
 can identify the stabiliser of $\del_v$ in $\so(t+1,s+1)$ with $\so(t+1,s+1)_{\del_v}=\so(t,s)\ltimes \rr^{t,s}$ and we get that 
\belabel{stabv}
\hol(\tem,\tg)\subset\so(t,s)\ltimes \rr^{t,s}
=\left\{ \left. \begin{pmatrix} 0& g(w,.) & 0 \\ 0& A & -w \\ 0&0&0\end{pmatrix} \; \right|  A\in \so(t,s), w\in \rr^{t,s}\right\},
\end{equation}
where the matrices are with respect to the splitting $T\tem=\rr\del_v\+TM\+\rr\del_u$ and the identification $T_pM = \rr^{t,s}$. 
With these identifications, there are two projections 
\[\pr_{\so(t,s)}:\hol(\tem,\tg)\to \so(t,s), \qquad \pr_{\rr^{t,s}}:\hol(\tem,\tg)\to \rr^{t,s},\] to the linear part $A$ and the translational part $w$ in (\ref{stabv}) of $\so(t+1,s+1)_{\del_v}=\so(t,s)\ltimes \rr^{t,s}$. Since derivatives of the curvature are contained in the holonomy algebra, \Cref{curvprop} implies that 
\belabel{projections}
 \begin{array}{rcl}
 \pr_{\so(t,s)}\left(
 \tnab_{\del_u}^q\tnab_{X_1}\cdots 
\tnab_{X_p}\widetilde{\RR})(Y,Z)
\right)
&=&
\frac{c}{u^q}(\nabla_{X_1}\cdots 
\nabla_{X_p}{\RR})(Y,Z)
\\
\pr_{\rr^{t,s}}
\left(
 \tnab^q_{\del_u}\tnab_{X_1}\cdots 
\tnab_{X_p}\widetilde{\RR})(Y,Z)
\right)
&=&
\frac{c}{u^{q-1}} \sum_{i=1}^p
(\nabla_{X_1}\stackrel{i\atop \uparrow}{\cdots }
\nabla_{X_p}\R)(Y,Z)X_i,
\end{array}\end{equation}
where $X_i,Y,Z\in T_pM$ and $c$ is a nonzero constant.

 A first description of the holonomy of $(\tem,\tg)$  was obtained in \cite{leistner05a}. This description is the first part of the following proposition.
\bs[{\cite[Theorem 4.2]{leistner05a}}] \label{myprop}
Let $\gg$ be a semi-Riemannian metric on $\M$  with holonomy algebra $\hol(\gg)$ and $\tg$ the metric 
$\tg=2\, \d u\, \d v +u^2\gg$ on $\rr^+\times \rr\times \M$. Then 
\[
\hol(\tg)\ \subset\  \hol(\gg)\ltimes\rr^{t,s}\ \subset\ \so(t,s)\ltimes\rr^{t,s}\ =\ \so(t+1,s+1)_{\del_v},\]
 and
 \[
\mathrm{pr}_{\so(t,s)}(\hol(\tg))\ =\ \hol(\gg).\]
Moreover, if $(M,g)$ admits a nonzero parallel vector field $X$, then 
\[\hol(\tg)\ \subset\  \hol(\gg)\ltimes X^\perp,\]
where $X^\perp\subset T_pM$ denotes the subspace orthogonal to $X_p$ with respect to $g$.
\es
\bprf
The proof of the first part of the proposition was given in \cite{leistner05a} and uses  equations~\re{uvconenabla} to  compute explicitly the parallel transport in $(\tem,\tg)$. Indeed, 
let  $\widetilde\gamma:[t_0,t_1]\to \tem$ be a piecewise smooth  curve given by $\widetilde\gamma(t)=\left(u(t), v(t) ,\gamma(t) \right) $ with a curve $\gamma$ in $\M$.  Let $Y(t)$ be a parallel vector field along $\gamma$ with respect to $\nabla$ and tangential to $\M$. Then one checks that
the vector field
\[ \widetilde Y(t) = \frac{1}{u(t)} Y(t) +    f(t) \cdot \del_v \]
is parallel with respect to $\tnab$ along $\widetilde\gamma$,  where
$f(t)= \int_{t_0}^{t} g_{\gamma(s)}\left( \dot \gamma(s),Y(s) \right) \d s$.
In particular,  the parallel transport of  $Z\in T_{(u(t_0),v(t_0),\gamma(t_0))}\M$ along $\widetilde\gamma$ is given by
\[\widetilde{\cal P}_{\widetilde\gamma} (Z)= \tfrac{1}{u(t_1)} \cal P_\gamma(Z)+  \left(\int_{t_0}^{t_1} \gg_{\gamma(t)}\left( \dot \gamma(t),\cal P_\gamma|_{[t_0,t]} (Z) \right) \d t \right)\del_v|_{\widetilde \gamma(t_1)}.\]
This implies that for a loop $\widetilde \gamma$ starting and ending at $(u,v,p)\in\tem$ we have that
\[\pr_{\so(t,s)}\left( \widetilde{\cal P}_{\widetilde\gamma}\right) = \tfrac{1}{u}\,\cal P_\gamma,\]
which shows that $\pr_{\so(t,s)}(\hol(\tg))=\hol(g)$.

For the second part, in the case where $(M,g)$ admits a parallel vector field $X$, the statement follows from the the Ambrose-Singer Holonomy Theorem and  the second equation in (\ref{projections}) as $(\nabla_{X_1}{\cdots }
\nabla_{X_p}\R)(Y,Z,X_i,X)=0$ for all $X_i\in TM$ if $X$ is parallel.\eprf
%
%
%

Note that this does {\em not} establish the inclusion $\hol(g)\subset\hol(\tg)$. 
Hence, for a metric of the form $\tg=2 \d u\d v+u^2g$ this result allows for the possibility that $\hol(\tg)$ is not completely determined by $\hol(g)$. Indeed, for the space of translations in $\hol(\tg)$,
\[T:=\hol(\tg)\cap \rr^{t,s}\] we have the following possibilities:
\bnum
\item $T=\rr^{t,s}$: 
In this case the holonomy of $\tg$ is completely determined by the holonomy of $g$ and we have $\hol(\tg)=\hol(g)\ltimes \rr^{t,s}$.
\item 
$T\not=\rr^{t,s}$: In this case we can distinguish two situations:
\bnum
\item $ \hol(g)\subset\hol(\tg)$: In this case there is a subspace of translations $T\subsetneq \rr^{t,s}$ such that 
$\hol(\tg)=\hol(g)\ltimes T$.
\item $\hol(g)\not\subset \hol(\tg)$.
\enum
\enum
In both cases in (2) it seems as if $\hol(g)$ does not determine $\hol(\tg)$ completely and that further knowledge about the geometry of $g$ is needed in order to decide whether (a) or (b) occur, to determine $T$, etc. In \Cref{algsec,holsec}  we will study these questions further, first purely algebraically and then using geometric properties of $\tg$. But first we will give some examples.

\subsection{Locally symmetric spaces and other examples}
\subsubsection{Locally symmetric spaces} Here we consider manifolds $(\tem,\tg)$ that arise via the construction (\ref{uvcone}) from semi-Riemannian locally symmetric spaces $(M,g)$. 
\begin{theorem} Let $(M,g)$ be a semi-Riemannian locally symmetric space, i.e., a semi-Riemannian manifold with $\nabla \R=0$. 
For $(M,g)$ we consider the metric 
$ \tg =2\d u\,\d v +u^2 \gg$  
on $\tem =  \rr^+\times  \rr\times \M$.
Then 
\[\hol_{\widetilde{p}}(\tem,\tg)=\hol_p(M,g)\ltimes T, \]
where $T= \hol_p(M,g) V$ with $V=T_pM$ and $\widetilde{p}=(1,0,p)\in \tem$. 
\end{theorem}

\bprf 
As a consequence of the Ambrose-Singer theorem and $\nabla \R=0$ we have that 
\begin{equation}\label{symcurv}
\hol_p(M,g)=\mathrm{span}\{\R(X,Y)|_p\mid X,Y\in T_pM\}.\end{equation}
The curvature $\widetilde{\R}$ of $(\tem,\tg)$ satisfies 
\eeqref{uvconeR}, which, together with \eeqref{symcurv}, shows that $\g=\hol(M,g)$ is contained in $\tgg=\hol_p(\tem,\tg)$.  Moreover,  by  \Cref{curvprop} we have that 
\[
(\tnab_{\del_u} \cdots \tnab_{\del_u}\tnab_{X_1}\widetilde{\RR})(X,Y)Z
=
\frac{c}{u^{q-1}}
 \R(X,Y,Z,X_1)\del_v,
 \]
for  nonzero constant $c$ and $X,Y,Z,X_1\in TM$, and all other derivatives of $\widetilde{\R}$ are zero.  This implies the 
claim. 
\eprf 
\bfolg\label{symfolg}  Let $(M,g)$ be a semi-Riemannian locally symmetric space, which is locally the product of
(non-flat) irreducible symmetric spaces. Then 
\[ \hol_{\widetilde{p}}(\tem,\tg)=\hol_p(M,g)\ltimes  T_pM.\]
\efolg
\bbsp\label{symexample} The following example shows that \Cref{symfolg}  does not extend
to indecomposable symmetric spaces such as the Cahen-Wallach space of dimension $n=m+2$,
\[ (M,g)=\left(\rr^n, g_{CW}=2\d x\d z +  \sum_{i,j=1}^{m}\lambda_{ij}y^iy^j \, \d z^2+\sum_{i=1}^{m} (\d y^i)^2\right), \]
where $(x,y^1, \ldots , y^m, z)$ are global coordinates on $\rr^{m+2}$ 
and where $S=(\lambda_{ij})$ is a constant symmetric matrix with $\det(S)\not=0$. 
In this case we have $\hol (M,g)=\rr^{m}\subset \so(1,m+1)_{\del_x}=\so (m) \ltimes \rr^{m}$ and  $T=\mathrm{span}(\del_x,\del_1\ldots, \del_{m})$ where $\del_i=\frac{\partial}{\partial y^i}$. We will explain these Lie algebras in more detail later on.
\ebsp

\subsubsection{pp-waves and plane waves} The pp-waves are Lorentzian manifolds that are generalisations of Cahen-Wallach spaces. Again we consider $M=\rr^n=\rr^{m+2}$ with global coordinates $(x,y^1, \ldots, y^m,z)$ and $f$ a function $f=f(y^1, \dots, y^m,z)$ of $ y^1, \dots, y^m$ and $z$ but not of $x$. Then a general pp-wave metric on $\rr^{m+2}$  is given by
\begin{equation}\label{ppmetric}
g=2\d x\d z+ 2 f(y^1, \dots, y^m,z) \d z^2+\sum_{i=1}^m(\d y^i)^2.\end{equation}
The  Levi-Civita connection and the curvature are  determined  by
\[
\nabla\partial_x=0,\quad \nabla_{\partial_i}\partial_j=0,\quad 
\nabla_{\partial_z}\partial_i=\partial_if \del_x,
\quad
\nabla_{\partial_z}\partial_z=\partial_z f \del_x- \sum_{i=1}^m\partial_if\, \del_i,
\]
and 
\[
\del_x\inter R=0,\quad R(\del_i,\del_j)=0,\quad
R(\del_i,\del_z,\del_z,\del_j)=-\del_i\del_j f.\]
In the basis 
$(\del_x,\del_1,\ldots, \del_m,\del_z-f\del_x)$
the metric is given by
\[\eta=\begin{pmatrix}0&0&1\\0&\1_m&0\\1&0&0\end{pmatrix}\]
and we can write  the curvature and its derivatives as endomorphisms in  $\so(\eta)$  as 
\begin{equation}\label{ppcurv}
(\nabla_{X_1} \ldots \nabla_{X_p}R)(\del_i,\del_z)=
\begin{pmatrix}
0&\left( X_1 \ldots X_p\del_i\del_j (f) \right)_{j=1}^m&0
\\
0&0&
-\left( X_1 \ldots X_p\del_i\del_j (f) \right)_{j=1}^m
\\
0&0&0
\end{pmatrix},
\end{equation}
where the $X_i$ are constant vector fields on $M=\mathbb{R}^n$. As for Cahen-Wallach spaces, their holonomy algebra contained in (and equal to, if the Hessian $\partial_i\partial_j f$ of $f$ is invertible) $\rr^m\subset \so(1,m+1)_{\del_x}$ and hence abelian.

Now we consider the semi-Riemannian manifold  $(\tem,\tg)$ of signature $(2,m+2)$ for a given pp-wave $(M,g)$ of dimension $n=m+2$. Then, by setting
\[A_{qrk_1\ldots k_si}:= (\tnab_{\del_u}^q
 \tnab_{\del_z}^r\tnab_{\del_{k_1}}\cdots 
\tnab_{\del_{k_s}}\widetilde{\RR})(\del_i,\del_z),\]
equations (\ref{projections}) in this case are
\belabel{projections1}
 \begin{array}{rcl}
 \pr_{\so(1,n-1)}\left(
A_{qrk_1\ldots k_si}
\right)
&=&
\tfrac{c}{u^q} \begin{pmatrix}
0&\left( \del_{k_1} \ldots \del_{k_s}\del_i\del_j \ f^{(r)} \right)_{j=1}^m&0
\\
0&0&
\vdots
\\
0&0&0
\end{pmatrix},
\\[2mm]
\pr_{\rr^{1,n-1}}
\left(
A_{qrk_1\ldots k_si}
\right)
&=&
\tfrac{c }{u^{q-1}} \Big( 
s
\del_{k_1}\cdots 
\del_{k_s}\del_i f^{(r)}
\del_x
+
\sum\limits_{j=1}^m
\del_{k_1}
\cdots 
\del_{k_s}\del_i\del_{j}f^{(r-1)}
\del_j
\Big).
\end{array}\end{equation}
where $f^{(r)}$ denotes the $r$-th derivative of $f$ with respect to the coordinate $z$.
This shows that $\hol(\tem,\tg)\subset \hol(M,g)\ltimes \del_x^\perp$, with $\del^\perp=\Span(\del_x,\del_1, \ldots ,\del_m)$,  as claimed in \Cref{myprop}. In general these projections could be coupled to each other, but for a special case we can say more:
\bs
Let $(M,g)$ be a pp-wave as in (\ref{ppmetric}) but with the condition that $f$ does not depend on $z$, i.e., $f=f(y^1,\ldots, y^n)$ and such that $\det(\del_i\del_j f)\not=0$ at one point 
(or, more generally, such that at one point 
\begin{equation} \label{cond:eq} \mathrm{span}\{\d (\partial_{k_1}..\partial_{k_p}f) \mid p\ge 1,\;  k_1, \ldots ,k_p\in \underline{m} \} = (\mathbb{R}^m)^*,\end{equation}
where 
$\underline{m} = \{ 1,\ldots ,m\}$). Then 
\[\hol(\tem,\tg)=\hol(M,g)\ltimes \del_x^\perp.
\]
\es
\bprf
We evaluate formulae (\ref{projections}) for $r=1$: since $f$ is independent of $z$, we have $f'=0$ and hence 
\[ \pr_{\so(1,m+1)}\left(
(
 \tnab_{\del_z}\tnab_{\del_{k_1}}\cdots 
\tnab_{\del_{k_p}}\widetilde{\RR})(\del_i,\del_z)
\right)=0,\]
 and 
 \[
 \pr_{\rr^{1,m+1}}
\left((
 \tnab_{\del_z}\tnab_{\del_{k_1}}\cdots 
\tnab_{\del_{k_p}}\widetilde{\RR})(\del_i,\del_z)
\right)
=
\sum_{j=1}^m
\del_{k_1}
\cdots 
\del_{k_p}\del_i\del_{j}f
\del_j
.\]
If $\det(\del_i\del_jf)\not=0$ (or if \re{cond:eq} holds at one point), this shows that $\Span(\del_1\ldots,\del_m) \subset \hol(\tem,\tg)\cap \rr^{1,m+1}$. This space however is not invariant under $\hol(M,g)$ and is mapped under the adjoint representation in $\hol(\tem,\tg) $ to $\rr\del_x$, so that $\hol(\tem,\tg)=\hol(M,g)\ltimes \del_x^\perp$.
\eprf
This proposition can be clearly generalised to functions $f$ that are  polynomial, say of degree $d$, in $z$ (and have arbitrary dependence on the $y^i$). It suffices to replace $r=1$ in the proof with $r=d+1$ and the condition on $f$ by the corresponding condition on $f^{(d)}$. It does not hold however for general $f$ as the following example shows:
\bbsp
Let $f(y,z)=\mathrm{e}^z y^2$ and $g$ a plane wave metric\footnote{Plane waves are pp-waves for which the function $f$ is a quadratic polynomial in the $y^i$'s  with $z$-dependent coefficients, i.e., 
$f(y^1,\ldots y^m,z)=\sum_{i,j=1}^m f_{ij}(z) y^iy^j$,  with $(f_{ij}(z))$ a symmetric matrix of functions of $z\in \rr$.}
 on $\rr^3$ defined by $f$,
\[
g=2\d x\d z+ \mathrm{e}^z y^2\d z^2+\d y^2.
\]
Its curvature and derivatives thereof are given by \eeqref{ppcurv} as follows
\[
\nabla_{\del_y}R=0,\qquad
(\nabla_{\del_z}^{(r)}R)(\del_y,\del_z)=
2 \begin{pmatrix}
0& \mathrm{e}^z&0
\\
0&0&
-\mathrm{e}^z
\\
0&0&0
\end{pmatrix}=: A(z),
\]
for all $r\ge 0$. Its holonomy algebra is one-dimensional. When we now consider the metric $\tg$,  formula (\ref{projections1}) shows that
\begin{eqnarray*}
(\tnab_{\del_u}^q
 \tnab_{\del_z}^r\widetilde{\RR})(\del_y,\del_z)
& =&
\tfrac{c}{u^q} \begin{pmatrix}
0&{ 2}u\,\mathrm{e}^z dy&0
\\
0&A(z)&-{ 2}u\,\mathrm{e}^z\del_y
\\
0&0&0
\end{pmatrix}, 
\\
(\tnab_{\del_u}^q
 \tnab_{\del_z}^r \tnab_{\del_y}\widetilde{\RR})(\del_y,\del_z)
 &=&
 \tfrac{{ 2}c }{u^{q-1}}\mathrm{e}^z \del_x, 
\end{eqnarray*}
with all other derivatives of the curvature being zero.
 Since $\tg$ is analytic, its holonomy  is determined by the derivatives of the curvature at a point, say at $v=x=y=z=0$ and $u=1$, and is spanned by the  two matrices
arising from $(\tnab_{\del_u}^q
 \tnab_{\del_z}^r\widetilde{\RR})(\del_y,\del_z)$ and $(\tnab_{\del_u}^q
 \tnab_{\del_z}^r \tnab_{\del_y}\widetilde{\RR})(\del_y,\del_z)$,
 \[
 {\small
 \begin{pmatrix}
 0&0&1&0&0
 \\
 0&0&1&0&0
 \\
  0&0&0&-1&-1 
 \\
  0&0&0&0&0
 \\
  0&0&0&0&0
\end{pmatrix},\qquad
\begin{pmatrix}
 0&0&0&1&0
 \\
 0&0&0&0&-1
 \\
  0&0&0&0&0 
 \\
  0&0&0&0&0
 \\
  0&0&0&0&0
\end{pmatrix}.}
\]
This shows that the holonomy of $\tg $ is abelian and is neither purely translational nor a semidirect sum of $\hol(g)$ with 
a Lie algebra of translations.
\ebsp
%
%
%
%
%

\subsection{Lift of parallel objects} In this section we analyse how parallel objects on $(M,g)$, such as vector fields and vector distributions, lift to $(\tem,\tg)$.
First we analyse how certain vector fields on $\M$ lift to $\tem$.
\blem\label{paralemma}
Let $\xi$ be a homothetic gradient vector field on $(\M,\gg)$, i.e., a vector field  with \belabel{kill}
\nabla \xi= a\ \mathrm{Id}\end{equation} for a constant $a\in \rr$ and such  that $\xi^\flat$ is not only closed but exact, $\xi^\flat=df$ for a smooth function $f$.
Then the vector field $\widetilde \xi$ defined by
\[\widetilde \xi =  f\,\del_v+\frac{1}{u} \xi  -a \del_u,\]
is parallel for $\tnab$. In particular, if $\xi$ is parallel for $(M,g)$, then $\widetilde \xi =  f\,\del_v+\frac{1}{u}\xi$ is parallel for $(\tem, \tg)$.
\elem
\bprf
First note that the condition \re{kill} implies that $\xi^\flat$ is closed, i.e., locally we can always  find a function $f$ such that $\xi^\flat=df$. 
Then we compute
\[
\tnab_{\del_u}\widetilde \xi=-\frac{1}{u^2} \xi +\frac{1}{u}\tnab_{\del_u} \xi =0,\]
because of \re{uvconenabla}. Moreover, we have for every $X\in TM$ that
\[
\tnab_{X}\widetilde \xi=df(X)\del_v +\frac{a}{u}X - \gg(\xi,X)\del_v - a\tnab_X\del_u =0,
\]
again by \re{uvconenabla} and $df=\xi^\flat$.
\eprf
In a similar way we can prove: 
\blem
\label{paralemma1}
Let $\L$ be a parallel null line bundle on $(\M,\gg)$.
Then the totally null $2$-plane bundle  $\P$ on $(\tem,\tg)$ spanned  by $\del_v$ and $\L$ is parallel for $\tnab$.
\elem
\bprf This follows from applying  \eeqref{uvconenabla} to a recurrent null vector field $\xi$ spanning $\L$ and  $\del_v$ being parallel for $\tnab$.\eprf
%
The following proposition will be used in \Cref{holsec} for the proof of \Cref{holtheointro}:
\begin{proposition}\label{recprop}
Let $(M,g)$ be a manifold with parallel null line bundle $\L$. Assume that the metric $\tg=2 \d u\d v+u^2g$ admits a recurrent vector field in the span of $\del_v$ and $\L$ that is not a multiple of $\del_v$. Then locally $g$ admits a parallel null vector field  in $\L$.
\end{proposition}
\begin{proof}
By \Cref{recurlemma} we can assume that $\L$ is spanned by a recurrent {\em gradient} vector field $\xi=\grad ( f)$, i.e., with $\xi^\flat=df$ and $\nabla \xi=\theta\otimes \xi$ with $\theta$ a multiple of $\xi^\flat$.
Then the vector field
\[\widetilde \xi=f\del_v+\frac{1}{u}\xi\]
satisfies \begin{eqnarray}\label{recurrent1}
\tnab_{\del_u}\widetilde \xi&=&0,
\\
\label{recurrent2}
\tnab_X\widetilde \xi
&=&
 \frac{1}{u}\theta(X)\xi\ =\ \theta(X)(\widetilde{\xi} - f\del_v),\quad \text{ for  }X\in T\M.\end{eqnarray}
Without loss of generality,  the assumption implies that $\tg$ admits a recurrent vector field of the form $\zeta=\widetilde{\xi}+h \del_v$ for a function $h$. It defines  a one-form $\alpha$ by $\tnab \zeta=\alpha\otimes \zeta$.
Then the fact that $\del_v$ is parallel and equation (\ref{recurrent1}) immediately show that
\[\del_uh=\alpha(\del_u)=\del_vh=\alpha(\del_v)=0.\]
Equation (\ref{recurrent2}) implies that
\[\tnab_X\zeta = \theta(X)\widetilde{\xi} +(dh(X)-f \theta(X))\del_v.\]
Hence the equation $\tnab\zeta =\alpha\otimes \zeta$ implies that $\alpha=\theta$ and 
\[\d h=(f+h)\theta.\]
Differentiating  this and taking into account that $\d f\wedge \theta = \d h\wedge \theta=0$ gives
\[
0
=
(f+h)\d \theta
\]
If $\d\theta\not=0$ this implies $h=-f$. This  contradicts the above $\d h=(f+h)\theta$, as it would imply that $h$ and hence $f$ are constant.
So we must have $\d\theta=0$. This however implies that one can rescale $\xi$ to a parallel null vector field.
\eprf

Finally, for parallel distributions of $(\M,\gg)$ we get
\blem
Let $W\subset T \M$ be a parallel distribution on $(\M,\gg)$. Then the distribution $\rr\partial_v\+ W\subset T\tem$ is parallel.
\elem
\bprf
The distribution $W$ is locally spanned by vector fields $W_1,\ldots , W_k$.
Then one checks that for $\widetilde W_i:=\del_v +\frac{1}{u}W_i$ we have
$\tnab_{\partial_u}\widetilde W_i=0$ and 
\[\tnab_{X}\widetilde W_i=-\gg(X,W_i)\del_v+{\frac{1}{u}}\nabla_XW_i\in \rr \partial_v\+W,\]
for all $X\in TM$.
\eprf

%
%
%
\section{Results about indecomposable subalgebras of $\so(t+1,s+1)$}
\label{algsec}
In this section we will prove several algebraic results about indecomposable subalgebras of $\so(t+1,s+1)$ stabilising a null line or a null vector. We will use these results in the next section when studying further  the holonomy of metrics of the form $\tg=2 \d u\d v+u^2 g$.
\subsection{Indecomposable subalgebras stabilising a null vector}\label{alg1sec}
In this section we will fix some notations and observe some fundamental facts about indecomposable subalgebras of $\so(t+1,s+1)$ stabilising   a null vector. In particular, in this section we will see why the vector space $Z^1(\g,V)$ of $1$-cocycles  of a Lie algebra $\g$ with values in a $\g$-module $V$ comes into play. Recall that 
\begin{equation}
\label{cycle}
Z^1(\g,V) :=\{\vf:\g^*\otimes  V \mid \vf([X,Y])=X\vf(Y)-Y\vf(X)\text{ for all }X,Y\in \g\}\end{equation}
and 
\[ H^1(\g , V) := \frac{Z^1(\g,V)}{\d V},\]
where
\[ \d : V \rightarrow Z^1(\g , V), \quad \d v (X) :=  Xv,\quad v\in V,\quad X\in \g .\]

Let $\widetilde V$ be a semi-Euclidean vector space of signature $(t+1,s+1)$ with metric $\widetilde g$ and let $\ev_{\pm}$ be two null vectors such that $\widetilde g(\ev_-,\ev_+)=1$. We split $\widetilde V=L_-\+ V\+ L_+$ with $L_\pm=\rr\cdot \ev_\pm$ and $V=(L_-\+L_+)^\perp$ which is equipped with the metric $g=\widetilde g|_{V\times V} $. With respect to this splitting the stabiliser of $L_-$ in $\so(\widetilde V)$, denoted by 
$\so(\widetilde V)_{L_-}$ is given as
\begin{eqnarray*} \so(\widetilde V)_{L_-}&=&(\rr\+\so(V))\ltimes V\\
&:=&
\left\{
\left. 
(a,X,v)
:=
\begin{pmatrix} a & -v^\flat &0 
\\0&X &v 
\\ 0&0&-a\end{pmatrix}
\ \right|\ 
a\in \rr, X\in \so(V), v\in V\right\} .\end{eqnarray*}
The action of $\so(\widetilde V)_{L_-}$ on $\widetilde V=L_-\+V\+L_+  \cong \mathbb{R} \+ V \+ \mathbb{R}$ is given by
\belabel{act}
(a,X,v)\cdot \begin{pmatrix}r\\ u\\ s\end{pmatrix}= \begin{pmatrix}ar- g(v,u)\\ Xu+sv\\ -as
\end{pmatrix}
.\end{equation}
Furthermore we record the formula for the Lie bracket in $\so(\widetilde{V})_{L_-}$:
\belabel{bracket}
\left[ (a,X,v),(b,Y,w)\right]
=
\left(
0, [X,Y], (X+a)w-(Y+b)v
\right).
\end{equation}

The stabiliser of the {\em vector} $\ev_-$ is given as $\so(\widetilde V)_{\ev_-}=\so(V)\ltimes V$, i.e., is obtained by requiring  $a$ to be zero in the above formulae. 
Note that, 
the adjoint  action of $\so(\widetilde V)_{\ev_-}=\so(V)\ltimes V$ preserves  the ideal $V$, whereas the linear action on $\widetilde V$ does not preserve the subspace $V\subset \widetilde{V}$.

Furthermore  note that there are natural projections
$\pr_V$ and $\pr_{\so(V)}$ on $V$ and $\so(V)$. For 
a subalgebra $\wt\g\subset \so(V)\ltimes V$
we call $\g:=\pr_{\so(V)}(\widetilde{\g})$ the {\em linear part of $\widetilde{\g}$} and 
 $T:=\wt\g\cap V$ the {\em translations in $\widetilde{\g}$}. Note that   
$\widetilde \g \subset \g\ltimes V$ but in general $\g \not\subset \widetilde{\g}$. 
\begin{proposition}
\label{algprop}
Let 
$\widetilde \g\subset \so(\widetilde V)_{\ev_-}=\so(V)\ltimes V$ be a subalgebra, 
$\g$ its linear part  and 
$T$
the translations in $\widetilde \g$.
Then
\begin{enumerate}
\item\label{ideal} $T$ is an ideal in $\widetilde \g$.
\item\label{inv} $T\subset V$ is invariant under $\g$, and consequently $\g$ acts on $V/T$.
\item\label{quot} We have an inclusion of Lie algebras $\widetilde \g/T\subset\g\ltimes V/T$.
\item\label{z1} There is a $ \vf\in Z^1(\g, V/T)$ such that  $\widetilde \g/T =\{ (X, \vf(X))\mid X\in \g\}$.
\item\label{z1comp} If $T$ has a $\g$-invariant complement $T'$, then there is a $\vf\in Z^1(\g,T^\prime)$,  such that
\[
\widetilde \g=\h_\vf\ltimes T,\quad \text{ where }\h_\vf=\{(X,\vf(X))\in \widetilde \g\mid X\in \g\}
\]
\enum
\end{proposition}
\bprf
\Cref{ideal,inv,quot} are obvious from the definitions. For \Cref{z1} we define $\vf(X)=v\mod T$ if $(X,v)\in \widetilde{\g}$. Since $(X,v)\in \wt\g$ and $(X,w)\in \wt\g$ implies that $v-w\in T$, this map is well defined. From \eeqref{bracket} we see that ${\vf}$ is an element in $Z^1(\g,V/T)$. Finally, \Cref{z1comp} follows easily from \Cref{z1}
using the identification $V/T=T'$ as $\g$-modules. 
\eprf

\btheo
\label{algtheo1}
Let 
$\widetilde \g\subset \so(\widetilde V)_{\ev_-}=\so(V)\ltimes V$ be a subalgebra acting indecomposably on  $\wt V$. Let
$\g\subset  \so(V)$ and $T\subset V$ be respectively the linear part and translational ideal of $\widetilde{\g}$. 
\bnum
\item 
If 
$T$ has a $\g$-invariant complement $T^\prime$ and 
$H^1(\g,V)=0$, then, up to conjugation in $\so(V)\ltimes V$, $\wt\g =\g\ltimes T$ and   $T^\perp$ is degenerate or zero. 
In particular, if $T$ is nondegenerate and $H^1(\g,V)=0$, then $T=V$. 
\item If $T$ is degenerate such that  $L=T\cap T^\perp$ is a null line (this is the case for example when $T$ is degenerate and $g$  Lorentzian) and if the representation of $\g$ on $V/L^\perp$ satisfies that 
  $H^1(\g,V/L^\perp)=0$, then $\g$ acts trivially on $L$ or, up to conjugation in $\so(V)\ltimes V$, $\wt\g$ preserves  $L$.
\enum
\etheo
\bprf
(1)
First assume $V=T\+T^\prime$ is a $\g$-invariant decomposition. In virtue of \Cref{algprop}, $\widetilde\g=\h_\vf\ltimes T$, for 
some $\varphi \in Z^1(\g, T')$.  Since $Z^1(\g,V)=\d V$ and $Z^1(\g,T^\prime)\subset Z^1(\g,V)$, we find a $v\in V$ such that 
\[\vf(X)=Xv,\]
for all $X\in \g$. Then every element  $(X,\vf(X))=(X,Xv)\in \h_\vf$ can be conjugated to $X$ by a conjugation with the translation given by $v$, i.e., with
\belabel{Aconj}
A_v=\begin{pmatrix}
1 & -v^\flat&-\frac{1}{2} g(v, v)
\\
0&\1& v
\\
0&0&1
\end{pmatrix}.
\end{equation}
Indeed, for each $X\in \g$ we get
\[
A_v 
\begin{pmatrix} 0 & - (X v)^\flat &0 
\\0&X &Xv 
\\ 0&0&0\end{pmatrix}
A_v^{-1}
=\begin{pmatrix} 0 & 0 & 0
\\0&X &Xv
\\ 0&0&0\end{pmatrix}A_{-v}
=\begin{pmatrix} 0 & 0 & 0
\\0&X &0 
\\ 0&0&0\end{pmatrix},
\]
using that $X\in \so(V)$.
This shows that after conjugation with a translation, we have that $\g\subset \wt\g$. 
Hence $\wt\g=\g\ltimes T$, where $T=\wt\g\cap V$. Note that this already implies that $T$ is nonzero, because
otherwise $\wt\g=\g \subset \so (V)$, which contradicts indecomposability. 
Since $T$ is $\g$ invariant, also the orthogonal complement $T^\perp$ of $T$ in $V$ is $\g$ invariant. Then \eeqref{act} shows that  $T^\perp\subset \widetilde{V}$ is also invariant 
under the action of $T\subset \so (\widetilde{V})$ on $\widetilde{V}$ 
and therefore $T^\perp$ is $\wt\g$-invariant.  Hence, by indecomposability of $\wt\g$,  $T^\perp$ has to be degenerate or zero.
 
(2) Assume that $T$ is degenerate such that $L:=T\cap T^\perp$ is a null line. By \Cref{inv} of \Cref{algprop},  $L$ is invariant under $\g$. Moreover, by \Cref{z1comp} of \Cref{algprop} we have that there is a $ \vf\in Z^1(\g, V/T)$ such that  $\widetilde \g/T =\{ (X, \vf(X))\mid X\in \g\}\subset \g\ltimes V/T$. Hence, if $\wt\vf:\g\to V$ is a lift of $\vf$ we can write $\wt\g=\h_{\wt\vf}+ T$, where $\h_{\wt\vf}=\{(X,\wt\vf(X))\in \g\ltimes V\mid X\in \g\}$. Note that, since $T$ may not have an invariant complement,  in general we do not have that $\wt\vf\in Z^1(\g,V)$ and neither that $\h_{\wt\vf}$ is a subalgebra.

Let $L^\perp$ be the hyperplane in $V$ that is orthogonal to $L$. It is $L\subset T\subset L^\perp$ and hence, 
by formula (\ref{act}), $L$ is annihilated by the translations $T$ in $\wt\g=\h_{\wt\vf}+ T$. 
It remains to show that $L$ is invariant under $\h_{\wt\vf}$, unless $\g$ acts trivially on $L$. For this we 
consider 
 the projection $\pi:V/T\twoheadrightarrow V/L^\perp$ and distinguish two cases:

{\bf Case 1:} $\pi\circ \vf :\g\to V/L^\perp$ is zero. This means that the image of the lift $\wt\vf$ is contained in $L^\perp$. This however implies that $L$ is not only invariant under $\g$ but also under $\wt\g=\h_{\wt\vf}+ T$. Indeed, from formula (\ref{act}) it follows for an element $(X,\wt\vf(X))\in \h_{\wt\vf}$ and $\ell\in L$, that $(X,\wt\vf(X))\cdot \ell= X\cdot \ell - g(\wt\vf(X), \ell)\e_-=X\cdot \ell \in L$, since $\wt\vf(X)\in L^\perp$ and $\g$ leaves $L$ invariant. Hence, in this case $L$ is $\tgg$-invariant.

{\bf Case 2:} $\pi\circ \vf :\g\to V/L^\perp$ is not zero, i.e., the image of $\wt\vf$ is not contained in $L^\perp$. In this case, similarly to (1), we try to find a conjugation with a translation that shows that $L$ is invariant under $\h_{\wt\vf}$ (after conjugation). For $v\in V$ to be determined, we consider the associated translation $A_v$  as in \eeqref{Aconj}.
 Then, as in (1), for an element 
 \[(X,\wt\vf (X))= 
 \begin{pmatrix} 0 & -(\wt\vf(X))^\flat &0 
\\0&X &\wt\vf(X)
\\ 0&0&0\end{pmatrix}
 \]
 we get that
 \belabel{conjugate}
 A_v(X,\wt\vf(X))A^{-1}_v
 =
 \begin{pmatrix} 0 & -(\wt\vf(X)-Xv)^\flat &0 
\\0&X &\wt\vf(X)-Xv 
\\ 0&0&0\end{pmatrix}.
\end{equation}

 Fix $\ell\in L$ and $\hat\ell\in V$ such that $g(\ell,\hat\ell)=1$. Then define  $0\not=\lambda\in \g^*$ and $\rho \in \g^*$  by $\wt\vf(X)=\lambda(X)\hat\ell  \mod L^\perp$ and $X \ell = -\rho(X)\ell$, for $X\in \g$. 
 This is summarised in  $(X,\wt\vf(X))\cdot \ell =-\lambda(X)\e_- - \rho(X)\ell$. 
It also implies  that $X\hat\ell = \rho(X)\hat\ell \mod L^\perp$, i.e., 
 $\rho:\g\to \gl(V/L^\perp)$ is the induced representation of $\g$ on $V/L^\perp$.
 If we assume that $\g$ does not act trivially on $L$, $\rho$ is not zero. The key observation now is that   $H^1(\g,V/L^\perp)=0$ implies that $\lambda=c \rho $ for a constant $c$. Indeed, $\varphi \in Z^1(\g,V/T)$ induces an element 
 $\overline{\varphi}\in Z^1(\g,V/L^\perp)$. So $H^1(\g,V/L^\perp)=0$ implies that  
$\overline{\vf}(X)=X(c\,\hat\ell\mod L^\perp)= cX \hat\ell\mod L^\perp =c\rho(X) \hat\ell\mod L^\perp$ and thus $\widetilde{\vf}(X)= c\rho(X) \hat\ell \mod L^\perp$.

Now, in \eeqref{Aconj} we set $v:=c\ \hat\ell$. Taking into account that $g(\hat\ell,\ell)=1$, formula (\ref{conjugate}) shows that
\[
A_v(X,\wt\vf(X))A^{-1}_v\cdot \ell =
-(\lambda(X)-c\rho(X)) \ev_--\rho(X)\ell =- \rho(X)\ell.\]
This shows that after conjugation with a translation the null line $L$ is invariant under $\h_{\wt\vf}$ and hence under $\wt\g$.
\eprf

\bbsp
Consider $\g=\rr^n\subset \so(n)\ltimes \rrn= \so(1,n+1)_{\e_0}$, where $\e_0\in\rr^{1,n+1}$ is a null vector. Then for $T=\rr\cdot \e_0$ one can check that $\tgg=\g\ltimes T\subset \so(2,n+2)_{\e_-}$ is indecomposable. Similarly, for $T=\mathrm{span}(\e_0, \ldots \e_{n})$, $\tgg=\g\ltimes T$ is indecomposable. Note that the latter is the holonomy algebra of a $(\tem,\tg)$ for a Cahen-Wallach space $(M,g_{CW})$ of dimension $n+2$ presented in  \Cref{symexample}.
\ebsp


\subsection{Indecomposable subalgebras with completely reducible linear part}
\label{vicentesec}

The main result of this section is  Theorem \ref{algtheo},
which is a generalisation to arbitrary signature of a result in \cite{bb-ike93}  for an indecomposable  stabiliser in $\so(1,n+1)$  of a null {\em vector}\footnote{We  point out that  in \cite{bb-ike93} a similar result for the stabiliser in $\so(1,n+1)$  of a null {\em line} is given.}. 
It gives a description of all indecomposable subalgebras $\widetilde \g\subset \so(\widetilde V)_{\ev_-}=\so(V)\ltimes V$  with completely reducible linear part and non-degenerate translational part.

The main results of this and  the next section use a result about Lie algebra cohomology\footnote{We do have self-contained proofs of Theorems \ref{algtheo} and  \ref{z1lemma} that do not use Theorem \ref{HSS}, but for the sake of brevity we do not present them here as they are longer.}, which we will present first. In the following, for a $\g$-module $V$, we denote by $V^\g$ the $\g$ invariant vectors,
\[V^\g=\{ v\in V\mid Xv=0\text{ for all }X\in \g\}.\] 

\begin{theorem}[{\cite[Theorem 13]{HochschildSerre53}, \cite[Theorem 2.28]{Solleveld02}}] \label{HSS}
Let $\g$ be a Lie algebra and $V$ a $\g$-module, both finite-dimensional and over a field $\mathbb{F}$ of characteristic zero. Assume that there is an ideal $\b$ in $\g$ such that
\begin{enumerate}
\item there is a subalgebra $\h$ in $\g$ such that $\g=\h\ltimes \b$. and 
\item $V$ and $\g$ are completely reducible as $\h$-modules.
\end{enumerate}
Then 
\[
H^p(\g,V)\simeq \sum_{i+j=p} H^i(\h, \mathbb{F})\otimes H^{j}(\b,V)^\g .\]
In particular, when $p=1$,
\begin{equation}\label{HS}
H^1(\g,V)\simeq H^1(\b, V)^\g +   (\h/[\h,\h])^*\otimes V^\g .\end{equation}
\end{theorem}
The original version of this theorem is due to Hochschild and Serre \cite[Theorem 13]{HochschildSerre53}, in which the existence of $\h$ was not assumed but that   $\g/\b$ is semisimple. Solleveld proved the generalisation that is given here in his Master's thesis \cite[Theorem 2.28]{Solleveld02}. Equation (\ref{HS}) for $p=1$ follows from the facts that $H^0(\h,\mathbb{F})=\mathbb{F}$, $H^0(\b,V)=V^\b$ and that $H^1(\h,\mathbb{F})$ is isomorphic to $(\h/[\h,\h])^*$.

Now we turn to the main result of this section. We use the same conventions as in \Cref{alg1sec}.

\btheo\label{algtheo}
Let 
$\widetilde \g\subset \so(\widetilde V)_{\ev_-}=\so(V)\ltimes V$ an indecomposable subalgebra
which satisfies the following  properties
\bnum
\item 
$\g=\mathrm{pr}_{\so(V)}(\widetilde\g)$
acts completely reducibly on $V$, and 
\item the translational ideal $T=\widetilde\g\cap V$ is non-degenerate.
\enum
Under these assumptions, let $\g=\z\+\g^\prime$ be the decomposition of $\g$ into its centre and the semisimple derived Lie algebra.
Then, $\g$ acts trivially on $T^\perp$ and $T\not=0$. Moreover, there is  a linear map $\vf:\g \to T^\perp$ with $\vf|_{\g^\prime}=0$  such that after  conjugation in $\so(V)\ltimes V$, $\widetilde \g$ is of the form 
$
\widetilde \g=\h_\vf \ltimes T
$, where 
\belabel{hdef}\h_\vf =\{ (X,\vf(X))\in \so(V)\ltimes V\mid X\in  \g\},\end{equation}
and the image of $\vf$ is co-null in $T^\perp$, i.e., $(\im\,\vf)^\perp\subset T^\perp$ is totally null. 
\etheo

\bigskip

The proof of this theorem is based on a lemma which will follow from Theorem \ref{HSS}. 
Since $V$ is a completely reducible module,  $\g$ is reductive and hence $\g=\z\+\g'$, where $\g'=[\g , \g ]$ is semisimple, $\z$ is the centre of $\g$ and  we denote the projection to $\z$ by $\pi_\z:\g\to\z$. 

\blem\label{L3}
Let $V$ be a semi-Euclidean vector space and $\mathfrak{g}\subset \so (V)$ be a 
Lie subalgebra which acts completely reducibly on $V$. 
Then 
\[ Z^1(\g , V) = \d V \oplus \iota (Z^1(\z , V^\g )),\]
where  $\z$ the center of $\g$ and $\iota:   Z^1(\z,V^\g)\to Z^1(\g,V)$ is the inclusion $\iota(\vf)=\vf\circ \pi_\z$ with $\pi_\z:\g\to \z$. In particular,
\[ H^1(\g , V) \simeq H^1(\z , V^\g ).\]
\elem

\begin{proof}
First note that for $\vf\in Z^1(\z,V^\g)=\mathrm{Hom}(\z,V^\g)$,   $\iota(\vf)$ is indeed a cocycle in $Z^1(\g,V)$. Moreover, with $V$ completely reducible we have  $V^\g\cap \g V=\{0\}$ and hence that \[\d V\cap  \iota(Z^1(\z,V^\g))=\{0\}.\]
It remains to show that
\[
H^1(\g,V) \simeq Z^1(\z,V^\g).\]
But we can apply Theorem \ref{HSS} to $\g$, $\b=\z$ and $\h=\g^\prime$ to get from equation (\ref{HS}) that 
\[ H^1(\g,V)\simeq H^1(\z, V)^\g.\]
Therefore it remains to show that $H^1(\z, V)^\g$ is isomorphic to $Z^1(\z , V^\g )$. We note that
\[H^1(\z,V)^{\g}=\{[\vf]\in H^1(\z,V)\mid \vf\in Z^1(\z,V): \forall X\in \g\ \exists v\in V: X \vf = \d_\z v\}.\]
Clearly, $Z^1(\z,V^\g)$ injects into $H^1(\z,V)^{\g}$ by mapping a cocycle to its equivalence class in $H^1(\z,V)^{\g}$, but we have to show that this is surjective.

For this, note that if $[\vf]\in H^1(\z,V)^{\g}$, we have that $\vf\in Z^1(\z,V)$ is such that for each $X\in \g$, there is a $v_X\in V$ such that 
\[X\vf(Z)=Zv_X.\]
This defines a linear map $\hat\vf:\g\to V/V^\z$ by the relation 
\[\hat\vf(X)= v_X+V^\z.\]
Since $Z\in \z$, it is 
\[Z v_{[X,Y]} =[X,Y]\vf(Z)=Z(Xv_Y-Yv_Y),\]
and so 
$\hat\vf$ is a cocycle, i.e., 
$\hat\vf\in Z^1(\g,V/V^\z) $. This induces a linear map 
\[ \Psi :H^1(\z,V)^\g \ni [\vf] \longmapsto [\hat\vf] \in H^1(\g,V/V^\z),\]
which clearly  has the kernel $Z^1(\z,V^\g)$. Therefore 
\[ H^1(\z,V)^\g/Z^1(\z,V^\g)\simeq \im(\Psi) \subset H^1(\g,V/V^\z).\]
Now we use again equation (\ref{HS}) in Theorem \ref{HSS} to get that 
\[H^1(\g,V/V^\z)\simeq H^1(\z,V/V^\z)^\g.\]
The last step in the proof is to show that $H^1(\z,V/V^\z)=\{0\}$. For this we set $W:=V/V^\z$ and we have to show that $Z^1(\z , W)=\d W$. The $\z$-module $W$ is an orthogonal sum of 2-dimensional indecomposable modules $W_i$ and 
$Z^1(\z , W) = \oplus_i Z^1(\z , W_i)$. Therefore we can assume without
loss of generality that $W=W_1$ is 2-dimensional. Let us denote by
$I$ a generator of the 1-dimensional Lie algebra $\so (W)$ such that $I^2=\epsilon \mathrm{Id}$, $\epsilon=\pm 1$. Then 
there exists $0\neq \lambda \in \z^*$ such that  $Xv=\lambda  (X)I v$ for all
$X\in \z$ and $v\in W$. Given $\varphi \in Z^1(\z , W)$, we have 
\[ 0=X\varphi (Y)-Y\varphi (X)= \lambda (X) I \varphi (Y) -\lambda (Y) I\varphi (X),\]
for all $X,Y\in \z$. 
The latter equation implies that there exists a vector $v\in W$ such that  
\[ I\varphi(X) = \lambda (X) v,\]
for all $X\in \z$.  This shows that $\varphi = \epsilon \lambda \ot Iv =  \epsilon \d v \in \d W$ and hence  that $H^1(\z,V/V^\z)=\{0\}$. 

This implies that $\im(\Psi)=\{0\}$ and hence that $Z^1(\z,V^\g)=H^1(\z,V)^\g\simeq H^1(\g,V)$. 
\end{proof}

Now we are in a position to prove \Cref{algtheo}:

\bprf[Proof of \Cref{algtheo}] From \Cref{algprop} we have that 
 $\wt{\g} = \mathfrak{h}_\vf \ltimes T$, where 
$\h_\vf$ is given by \eeqref{hdef} with  $\varphi \in Z^1(\g , T^\perp )$. It remains to verify that $\vf|_{\g'}=0$. Lemma \ref{L3} shows that,  up to conjugation of $\wt{\g}$ in $\so (V) \ltimes V$ by a translation in $T^\perp$ we have
$\varphi \in \iota (Z^1(\z , T^\perp \cap V^\g ))$. This shows that $\varphi$ vanishes on $\g'$ and takes values in 
$T^\perp \cap V^\g$. The $\g$-invariant decomposition
\[ T^\perp = (T^\perp \cap V^\g)\stackrel{\perp}{\oplus} \g T^\perp\]
shows that the subspace $\g T^\perp \subset V$ is non-degenerate.  Let us check
that it is  not only  invariant under $\g$ but also under $\wt{\g}$. For this is suffices 
to observe that, by our description of $\wt{\g}$ and the fact that $\mathrm{im}\, \varphi \subset T^\perp \cap V^\g$, the translational part of any element of $\wt{\g}$ is contained in $(T^\perp \cap V^\g) \oplus T$. Therefore it is perpendicular to $\g T^\perp$, which shows that $\g T^\perp \subset V \subset \widetilde V$ is
$\wt{\g}$-invariant. Since $\wt{\g}$ is indecomposable this proves that $\g T^\perp =0$. 

Note that this implies that $T\not= 0$, because otherwise $T^\perp=V$ and hence $\g=0$ and $\tgg=T=0$, which contradicts the indecomposability of $\tgg$.

Finally, let $(\mathrm{im}\, \varphi)^\perp$ be the  orthogonal space of $ \mathrm{im}\, \varphi$ in $T^\perp$ and  $W$ be
a $\g$-invariant complement of $\mathrm{im}\, \varphi \cap  (\mathrm{im}\, \varphi)^\perp$ in $ (\mathrm{im}\, \varphi)^\perp$. Then $W$ is non-degenerate. Again it is not only 
$\g$-invariant but  also $\wt{\g}$-invariant because the translational part of  any element in $\tgg$ is contained
in $(\mathrm{im}\, \varphi) \oplus T$ and 
$W\subset  (\mathrm{im}\, \varphi)^\perp\subset T^\perp$.
 Since $\wt{\g}$ is indecomposable this shows that $W=0$ and, hence, that 
$ (\mathrm{im}\, \varphi)^\perp \subset  \mathrm{im}\, \varphi$.
\eprf

\subsection{Cohomology of indecomposable subalgebras in $\so(1,n+1)$}
In this section we compute the $1$-cocycles for  subalgebras $\g$ of $\so(1,n+1)$ that act indecomposably on $V=\rr^{1,n+1}$.
Such a subalgebra is either irreducible, in which case it is equal to $\so(1,n+1)$  \cite{olmos-discala01} and hence $H^1(\g , V)=0$, or admits a parallel null-line $L=L_-=\rr\e_-$.
That such a subalgebra  belongs to one of the four types discussed in the proof of \Cref{z1lemma} below, was proven in \cite{bb-ike93}.

In the following we will use equations \re{act} and \re{bracket} 
and  the identifications in Section \ref{alg1sec}  with $(\wt{V}, \wt{g},V, g)$ replaced by 
$(V, g,V_0, g_0)$. Note that $g_0 = g|_{V_0\times V_0}$ is the standard Euclidean scalar product on $V_0=\mathbb{R}^n$. 
We will use the standard decomposition $V=\rr\cdot\e_-\+V_0\+\rr\cdot\e_+$ and the notation 
$\g_0=\pr_{\so(V_0)}(\g)$, $\g_0^\prime=[\g_0,\g_0]$, $\z=\z(\g_0)$ for a subalgebra $\g \subset \so (V)_{L}$.

\btheo\label{z1lemma}
Let $V=\rr\cdot\e_-\+V_0\+\rr\cdot\e_+$ be the Minkowski space  with null vectors $\e_{\pm}$ and Euclidean vector space $V_0$,  and let $\g\subset \so (V)_L\subset \so(V)$ be an indecomposable subalgebra.  Then 
\[ H^1(\g,V)=0,\]
or $\g$ annihilates $\ev_-$.
%
%
\end{theorem}

\begin{proof}
First note that if $\dim( V)=2$, i.e., $V_0=0$, then \[\g=\so(1,1)=\rr \begin{pmatrix} 1&0\\0&-1\end{pmatrix}, \] and
$H^1(\g,V)$ clearly is trivial. 

If $\dim(V)\ge 3$, then
according to \cite{bb-ike93}, any indecomposable subalgebra $\g$ of $\so(V)_L$, belongs to one of  four different types. 
Two of them annihilate $\e_-$, whereas the other two act {non-trivially} on $\rr\cdot \e_-$. The latter are given as follows, where
 $\z$ denotes the centre of $\g_0=\z\+\g_0^\prime$ with $\g_0^\prime=[\g_0,\g_0]$ semisimple:
\bnum
\item
$\g=(\rr\+\g_0)\ltimes V_0$. We can set \[\b:=(\rr\+\z)\ltimes V_0.\] Then $\g/\b=\g_0^\prime$ is semisimple and acts completely {reducibly on $V=\rr\e_-\+V_0\+\rr\e_+$.}
\item $\g=(\h_f\+\g_0^\prime)\ltimes V_0$, with $0\neq f\in \z^*$ and $\h_f=\{(f(Z),Z)\mid Z\in \z\}\subset \rr\+ \z$. Here we set 
\[\b:= \h_f\ltimes V_0,\] so that $\g/\b=\g_0^\prime$ {acts again completely reducibly on $V$.}
 \enum
Now we apply Theorem \ref{HSS} to $\g$, the ideal $\b$ as given in the above and $\g_0^\prime=\g/\b$. Since $\g_0^\prime$ is semisimple,  the second summand in (\ref{HS}) vanishes  and we get
\[
H^1(\g,V)\simeq H^1(\b, V)^\g.\]
In order to determine $H^1(\b, V)$ we can apply  Theorem  \ref{HSS} again, this time to $\b$, the ideal $\a=V_0$ and the subalgebra  $\h=\rr\+\z$ in case (1) and $\h=\h_f$ in case (2). In both cases $\h$ is abelian and acts completely {reducibly} on $\b$ and on $V$, so the assumptions of Theorem \ref{HSS} are satisfied and we
get 
\[
H^1(\b,V)\simeq  H^1(\a, V)^\b+\h^*\otimes V^\b .\]
Since  for both types of $\g$, $\b$ scales $\e_-$ and contains  {$\a=V_0$,} we have that {$V^\b=\{0\}$, cf.~(\ref{act}). So} it remains to show that $(H^1(\a, V)^\b)^\g=H^1(\a, V)^\g$ is trivial. Even though $\a=V_0$ is abelian, we cannot  apply Lemma \ref{L3}  to find $H^1(\a,V)$, because $\a$ does not act completely {reducibly} on $V$. 
Instead, we first note that
if $\dim(\a)=\dim(V_0)=1$, then $Z^1(\a,V)=  \a^*\ot V$, $\d V = \a^*\ot (\mathbb{R} \mathrm{e}_- \oplus V_0)$ and the line $\a^*\ot \mathrm{e}_+ \subset Z^1(\a,V)$ projects isomorphically 
onto $H^1(\a,V)$.  From (\ref{act}) we see that the action of an element $(a,X,v)\in \g$ on $H^1(\a,V)$ is given by multiplication with $-a$. Since  for both types of $\g$
there are elements with $a\neq 0$, we conclude that $H^1(\a,V)^\g=0$. Thus 
we can assume that $\dim(V_0)\ge 2$. Then $\vf\in Z^1(\a,V)$ splits into components $\vf=(\vf_-,\vf_0,\vf_+)$ with respect to $V=\rr\cdot \e_-\+V_0\+\rr\cdot \e_+$ and {with} $\vf_\pm\in \a^*$ and $\vf_0\in V_0^*\otimes V_0$. 
From (\ref{act}) we see that  ${u}\in V_0=\a$ acts on ${(v_-,v,v_+)\in V}$ as 
\[u\cdot (v_-,v,v_+)= (-u^\top v, v_+ u,0).\]
Since $\a$ is abelian,  the cocycle condition  for $\vf$ yields
\[ {u^\top\vf_0(v) - v^\top\vf_0(u)}=0, \quad \vf_+(u) v-\vf_+(v) u=0, \]
for all $u,v\in V_0$.
Since $\dim (V_0)\ge 2$ the second equation implies that $\vf_+=0$.  The first equation implies that $\vf_0$ is a symmetric endomorphism of $V_0$. 
This shows that $Z^1(\a,V)=V_0^*\+ S(V_0)$ and that  
\[H^1(\a,V)\simeq S_0(V_0),\]
where $ S(V_0)$ and $ S_0(V_0)$ denote the symmetric and the symmetric trace-free endomorphisms of $V_0$.  Hence, every element $[\vf]\in H^1(\a,V)^\g$ can be represented by a  symmetric trace free-matrix $S$. Therefore the equation that $[\vf]$ is $\g$-invariant, which {means}  that for every $(a,X,v)\in \g$ there  is a $(w_-,w,w_+)\in V$ such that  
\[ 
(a,X,v)\cdot \vf
=
\d (w_-,w,w_+)
 \]
 becomes,   by (\ref{act}),
 \[
 (a,X,v) (\vf(u) )- \vf([(a,X,{v}),(0,0,u)])
=
\begin{pmatrix}
{-v^\top S u}
\\
[X,S]u-aSu
\\
0
\end{pmatrix}
=
\begin{pmatrix}
 {-w^\top u}
 \\
 w_+ u
 \\
 0\end{pmatrix}
\] 
for all $u\in V_0$.
This implies that 
\[
[X,S] =(w_+\Id +aS).
\]
Taking the trace yields $w_+=0$  and multiplying both sides by $S$ and taking the trace
gives
\[a \,\trace(S^2)= \trace( [X,S]S) =0.\]
Since we can chose $a\not =0$ for both types, this implies that  $\trace(S^2)=0$. With $S$  symmetric, we obtain that $S=0$, hence $H^1(\a,V)^\g=\{0\}$ and consequently that $H^1(\g,V)=0$.
 \end{proof}
 
\bbem \label{types34} Similar arguments  can be used to determine $H^1(\g,V)$ for the other two types of indecomposable subalgebras of $\so(V)_L$, those that leave invariant the null vector $\e_-$ (notations as in \Cref{z1lemma}, for details about these subalgebras see \cite{bb-ike93}).
One of them is of the form $\g=\g_0\ltimes V_0$ and by applying to above arguments to 
$\b:=\z\ltimes V_0$
one can show that 
\[H^1(\g,V)=S_0(V_0)^{\g_0}\+ \z^*\+(V_0^{\g_0})^*,\]
where $S_0(V_0)^{\g_0}$ denotes 
the trace-free, symmetric matrices that commute with $\g_0$. 

A similar statement holds for the remaining fourth type where 
$\g=(\h_f\+\g_0^\prime)\ltimes T_0$, with $0\not=T_0\subsetneq V_0$ invariant under $\g_0$ such that $T_0^\perp\subset \ker(\g_0)^\perp$  and \[\h_f=\{(0,Z,f(Z))\mid Z\in \z\}, \text{ with $ f:\z\to T_0^\perp$ surjective.}
\]
Here one can apply the above strategy to 
$\b:=
 \h_f\ltimes T_0$.
 However, since the result is somewhat technical and we do not need it for what follows, we will not give the details here.
\ebem
Finally we study the two types of indecomposable  subalgebras of $\so(1,n+1)$ that stabilise the null line $L$ but act non trivially on $L$, i.e., the types considered in the previous theorem.
\bs\label{cycleprop}

Let $V=\rr\cdot\e_-\+V_0\+\rr\cdot\e_+$ be the Minkowski space with null vectors $\e_{\pm}$,  and let $\g\subset \so(V)_L\subset \so(V)$ be an indecomposable subalgebra stabilising a null line $L=\rr\e_-$ but acting  non trivially on $L$. Let 
$\rho\in \g^*$ be defined by the representation of $\g$ on $V/L^\perp$, i.e.,
\[(a,X,v) [u]=\rho(a,X,v)[u], \quad i.e.,\ \rho(a,X,v)=-a,\]
(according to formula (\ref{act})).
Then, every $\vf\in Z^1(\g,V/L^\perp)\subset\g^*$ is a multiple of $\rho$, or equivalently, $Z^1(\g,V/L^\perp)=d(V/L^\perp)$, i.e., $H^1(\g,V/L^\perp)=\{0\} $).
\es
\bprf
First we consider the type $\g=(\rr\+\g_0)\ltimes V_0$. Note that we do not exclude the case $V_0=0$, for which $\g=\so(1,1)$.   For $a\not=0$, every $\vf\in Z^1(\g,V/L^\perp)$ satisfies 
\[
0=\vf( \left[(a,0,0),(0,X,0)\right])=-a\vf(0,X,0),
\]
for all $X\in \g_0$. Hence $\vf|_{\g_0}=0$.
Similarly, we get
\[
a\vf(0,0,v)=\vf(\left[(a,0,0),(0,0,v)\right])
=
-a\vf(0,0,v),
\]
for all $v\in \rrn$. Hence $\vf|_{V_0}=0$. This implies that $\vf$ is a multiple of $\rho$.

Now we assume that $\g = (\mathbb{R} \zeta_0 \oplus \k)\ltimes V_0$,  
where $\k = \ker f \oplus \g_0' \subset \z \oplus \g_0'=\g_0 = \pr_{\so(n)}\g$, $f\in \z^*$, $\zeta_0 = (1,Z_0)$ and $Z_0\in \z$ is a vector in the centre $\z$ of $\g_0$ 
such that $f(Z_0)=1$. In particular, $\dim(V_0)\ge 2$. 
For $X\in \k$ we obtain
\[
0=\vf( \left[\zeta_0,(0,X,0)\right])=-\vf(0,X,0),
\]
i.e., $\vf|_{\k}=0$.
Moreover, for all $v\in \rrn$ from the cocycle condition we get 
\begin{eqnarray*}
-\vf(0,0,v)&=&
\vf\left(\left[ \zeta_0,(0,0,v)\right]\right)
\\
&=&
 \vf(0,0,(1+Z_0)v)
\ =\
 \vf(0,0,v)+\vf(0,0,Z_0v),
\end{eqnarray*}
i.e., that
\belabel{zeqn}
\vf(0,0,Z_0v)
=
-2\vf(0,0,v). 
\end{equation}
Applying \eeqref{zeqn} twice one obtains 
\[
\vf(0,0,Z_0^2v)
=
-2\vf(0,0,Z_0v)
=
4\vf(0,0,v).
\]
Since $Z_0\in \so(n)$, its square $Z^2_0$ is diagonalisable with only nonpositive eigenvalues. 
Hence we get that $\vf|_{V_0}=0$. This implies that $\vf$ is a multiple of $\rho$.
\eprf

\section{Holonomy of metrics $\tg=2\,\d u\,\d v+u^2g$}
\label{holsec}
In this section we will use the geometric lifting properties of metrics of the form $\tg=2\,\d u\,\d v+u^2g$  derived in Section \ref{uvgsec} and the algebraic results of Section \ref{algsec} in order study the holonomy of $\tg$.
For cones over manifolds $(M,g)$ of arbitrary signature but with completely reducible holonomy, Theorem \ref{algtheo} has the following consequences.
\bfolg\label{cor1}
Let $\gg$ be a semi-Riemannian metric of signature $(t,s)$ on a manifold $\M$ the holonomy algebra $\hol(\gg)$ of which acts completely reducibly. Consider the metric \[\tg=2\, \d u\, \d v +u^2\gg\] on $\tem = \rr^+\times \rr\times \M$ and assume that the holonomy $\widetilde \g :=\hol(\tg)$ of $\tg$ acts indecomposably, i.e.~without a proper non-degenerate invariant subspace, 
and that the translational ideal $T:= \tgg \cap V$ is non-degenerate.
Then 
\[
\hol(\tg)=\hol(\gg)\ltimes  V.\]
\efolg
\bprf First Proposition \ref{myprop} gives that $\g=\mathrm{pr}_{\so(t,s)}(\widetilde\g)=\hol(\gg)$.
Then Theorem \ref{algtheo} applied to $\widetilde \g$ shows that $\g T^\perp=0$. 
If $T^\perp\not=\{0\}$, then  $g$ admits a non-degenerate parallel vector field which, according to Lemma \ref{paralemma}, would lift to a non-degenerate parallel vector field for $\tg$. This is excluded by the assumption of indecomposability of $\tg$. 
\eprf

As an aside, let us record the consequence of \Cref{algtheo} for Lorentzian metrics of the form $\tg=2 \d u\d v+u^2g$. We have obtained this result in \cite[Section 9]{acgl07}.
\bfolg
Let $\gg$ be a Riemannian metric in dimension $n$ and $\tg=2\, \d u\, \d v+u^2\gg$ a Lorentzian metric. If the holonomy of $\tg$ acts indecomposably, then 
\[\hol(\tg)=\hol(\gg)\ltimes \rr^{n}.\]
\efolg
In the main result of this section we deal with metrics $\tg$ over {\em Lorentzian} metrics $g$.
\btheo\label{lortheo}
Let $\gg$ be a Lorentzian metric on an $n$-dimensional simply connected manifold $M$ and $\tg=2\, \d u\, \d v+u^2\gg$ of signature $(2,n)$ on $\rr^+\times \rr\times M$.  If the holonomy of $\tg$ acts indecomposably, then 
\[\hol(\tg)=\hol(\gg)\ltimes \rr^{1,n-1},\]
or $g$ admits a parallel null vector field and $\tg$ admits two linearly independent parallel null vector fields that are orthogonal to each other.
%
\etheo
\bprf
Set $\widetilde \g:=\hol(\tg)$,
 $\g:=\hol(\gg)$ and $V:=\rr^{1,n-1}$. Let $T=\widetilde\g\cap V$ be the pure translations in $\widetilde\g$.  We have to show that $T=V$, in which case we have that $\widetilde\g=\g\ltimes V$, or that $\g$ admits an invariant null vector.
Hence we assume from now on that $T\not= V$.
By Proposition \ref{myprop} we have that  $\widetilde\g \subset \g\ltimes V$ with $\g = \mathrm{pr}_{\so(1,n+1)}(\widetilde \g)$ and  $T$ is $\g$ invariant.  
 
 Since $\g$  is a holonomy algebra, we can apply the Wu splitting theorem and obtain $\g =\g_1\+\ldots \+\g_k$ and 
\[V=\rr^{1,n-1}=V_0\+^\perp V_1\+^\perp V_2\+^\perp \ldots \+^\perp V_k,\]
with $\g_i$ acting trivially on $V_j$ for $i\not= j$, all the $V_i$'s are non-degenerate, 
with $V_0$ a trivial representation and $V_i$ indecomposable for $i=1, \ldots , k$. Since we 
assume that $\widetilde \g$ acts indecomposably, $\tg$ does not admit non-degenerate parallel vector fields. Therefore, Lemma \ref{paralemma} implies that $V_0=\{0\}$.
Hence we can choose the $V_i$ in a way that $V_1$ is the Minkowski space and indecomposable for $\g_1$ and the remaining $V_i$ are Euclidean and irreducible for $\g_i$.
Note that for $i=2, \ldots, k$ we have that  $\g_i\subset \so(n_i)$, where $n_i=\dim(V_i)$. 
Moreover, we can write
\[
\g\ltimes V= (\g_1\+\ldots \+\g_k)\ltimes (V_1\+\ldots\+ V_k) =
(\g_1\ltimes V_1)\+ \ldots \+ (\g_k\ltimes V_k).\]

Not only $T$ but also
$T_i=\widetilde\g\cap V_i$ is $\g$-invariant. Hence we have for $i=2, \ldots , k$ that $T_i=\{0\}$ or $T_i=V_i$, and that $T_1$ is degenerate, trivial or equal to $V_1$. The same holds for $P_i=\pr_{V_i}T$ containing $T_i$. 

%

Since  $V_1$ is indecomposable but not necessarily irreducible, we have to consider several cases for $T$:

\medskip

{\bf Case 1:} $T$ is indefinite, i.e., of signature $(1, \dim(T)-1)$. In this case we have that $T\cap V_1=V_1$ and that   $T^\perp$ is positive definite and hence a direct sum of irreducibles that can be arranged such that $T^\perp=V_{\ell+1}\+\ldots\+ V_k$ with $ 1\le \ell \le k-1$ (recall that $T\not=\{0\}$ and that we are working under the assumption $T\not=V$). We apply  Theorem \ref{algtheo} to the following data: 
 
We define $\widetilde W:=\rr\e_-\+T^\perp \+\rr\e_+$ and a representation  $\rho :\widetilde\g \to \so(\widetilde W)_{\e_-}$ by 
$\rho(X,v)=(X|_{T^\perp},\pr_{T^\perp}(v))$. 
Since $T^\perp $ is positive definite, 
 it is $T\cap T^\perp=\{0\}$, so by  its very definition $\rho(\widetilde\g)$ satisfies that $\rho(\widetilde\g)\cap T^\perp=\{0\}$. 
 On the other hand, 
 $\rho(\widetilde\g)$ satisfies the assumptions of Theorem  \ref{algtheo}. Hence, with $\rho(\widetilde\g)\cap T^\perp=\{0\}$, the projection of $\rho(\tgg)$ onto $\so(T^\perp)$ acts trivially on $T^\perp$. But this contradicts the fact that $T^\perp=V_{\ell+1}\+\ldots \+ V_k$, where the $V_i$'s are irreducible for $\pr_{\so(1,n-1) }(\tgg)$ and hence for $\pr_{\so(T^\perp)}(\rho(\tgg))$.

\medskip

{\bf Case 2:} $T$ is positive definite (including the case $T=0$), i.e.,  $T\cap V_1=\{0\}$ in virtue of the indecomposability of the $\g_1$-module $V_1$. In this case  $T^\perp$ is non-degenerate and $V_1\subset T^\perp$, i.e., \[T^\perp= 
V_1\+\ldots \+V_\ell\quad\text{ and }\quad T=V_{\ell+1}\+\ldots \+ V_k.\] Set \[\g_-=\g_1\+\ldots \+\g_\ell\quad \text{ and }\quad \g_+=\g_{\ell+1}\+\ldots \+\g_k,\] where $\g_+=\z_++\g_+'$ is reductive with centre $\z_+$ and derived algebra $\g_+^\prime$, and $\g_1$ is either irreducible or indecomposable but with an invariant null line $L$. 

In the case when $\g_1$ acts irreducibly on $V_1$, $\g$ acts completely reducibly on $V$ and, since $T$ is positive definite, we can apply \Cref{cor1} to get a contradiction to $T\not=V$. 

Hence we can assume that $\g_1$ is
contained in the stabiliser of the null line $L$, i.e.,  $\g_1\subset \so(V_1)_L$.
Since $\g_+$ acts trivially on $T^\perp$ and the $V_i$'s are irreducible for $i=\ell+1,\ldots , k$,  and $\g_-$ acts trivially on $T$, we have that 
\begin{equation}\label{kerg+}
V^{\g_-}\cap T^\perp = V^\g
\end{equation}
As in   \Cref{algprop}, there is a $\vf \in Z^1(\g,T^\perp)$,  such that $\widetilde\g=\h_\vf \ltimes T$. 
Then for $X_\pm\in \g_\pm$ we have
\[
0
=\vf([X_+,X_-])= X_-\vf(X_+).
\]
Hence, using equality (\ref{kerg+}), we obtain $\vf(\g_+)\subset V^{\g_-}\cap T^\perp=V^\g$. If $\vf|_{\g_+}\not=0$, we conclude that  $V^\g$ is a non-trivial subspace of $T^\perp$ and thus $V^\g=L$. Hence, if $\vf|_{\g_+}\not=0$ there is a non-zero vector in $L$ that is annihilated by $\g$ and therefore   the metric  $g$ admits a parallel null vector field.  

Hence, for Case 2  we can assume that $\vf|_{\g_+}=0$ and  are left with  \[\vf:\g_-\longrightarrow  T^\perp=V_1\+\ldots \+V_\ell.\] 
Then for $X_i\in \g_i$ and $X_j\in \g_j$, with $i, j\in \{ 1,\ldots ,\ell\}$, and $i\not= j$  we have
\[0=X_i\vf(X_j)-X_j\vf (X_i),\]
and hence 
\begin{equation}\label{fij}
X_i\vf(X_j)=0.\end{equation}
Since the $V_{j\ge 2}$ are irreducible, this relation for $j=1$ implies that 
\[\vf|_{\g_1}\in Z^1(\g_1,V_1).\]
On the other hand, for $j\ge 2$ we have that 
\[
\vf|_{\g_j}\in Z^1(\g_j,L\+V_j),\]
where $L$ is the $\g$-invariant null line.
If we write $\vf=\vf_1+ \ldots +\vf_\ell$ with $\vf_i:\g_-\to V_i$, then relation (\ref{fij}) implies  that if there exists $X_j\in \g_j$ for some
$j\ge 2$ such that  $\vf_1(X_j)\not= 0$, and thus $\vf_1(\g_j)=L$, then $\g_1$ and hence $\g$ acts trivially on $L$. The latter case implies again that the metric  $g$ admits a parallel null vector field.

Hence, we have obtained that $g$ admits a parallel null vector field or that 
$\vf=\vf_1+ \ldots + \vf_\ell$  with $\vf_i\in Z^1(\g_i,V_i)$ for $i=1, \ldots \ell$. Since the $V_i$ for $i\ge 2$ are irreducible, we have that $Z^1(\g_i,V_i)=\d V_i$, by \Cref{L3}. The case  $i=1$ is covered by \Cref{z1lemma}  where we have shown that  $H^1(\g_1,V_1)= 0$ whenever $g$ does not admit a parallel null vector field. Hence, if $g$ does not admit a parallel null vector field we obtain from (1) in  \Cref{algtheo1} that $T^\perp$ is degenerate or zero. But this contradicts $T\not=V$ and that $T^\perp$ in Case 2 is non-degenerate.

\medskip

{\bf Case 3:} $T$ is degenerate, i.e., there is a $\g$-invariant null line $L=T\cap T^\perp$. Our aim is to apply point (2) in \Cref{algtheo1} and \Cref{cycleprop}. First note that $\g$ and therefore the indecomposable subalgebra $\g_1 \subset \so(V_1)$ both leave $T$ and hence the null line $L$ invariant. If $\g_1$ acts trivially on $L$, then $\g$ acts trivially on $L$ and the metric $g$ admits a parallel null vector field. Therefore we can assume that $\g_1$ does not act trivially on $L$. This means that we can apply \Cref{cycleprop} 
to $\g_1$ and $L^\perp \cap V_1$ to get that 
\[Z^1(\g_1,V_1/(L^\perp\cap V_1))=d( V_1/(L^\perp\cap V_1)).\]
On the other hand,  we note that there is a canonical identification
\[V/L^\perp \simeq V_1/(L^\perp\cap V_1),\]
which shows that 
$\g_2\+\ldots \+\g_k$ acts trivially on $V/L^\perp$. Hence,  
\[
Z^1(\g,V/L^\perp)=Z^1(\g_1, V_1/(L^\perp\cap V_1))=  d(V/L^\perp).\] Since we have assumed that $\g$ does not act trivially on $L$, (2) in  \Cref{algtheo1} implies that, up to conjugation, $\wt\g$ leaves invariant a null line $L$. This means that $(\tem,\tg)$ admits a recurrent null vector field in the span of $\del_v$ and $L$ (even a recurrent section in $L$). But in this situation, \Cref{recprop} ensures the existence of a parallel null vector field on $(M,g)$. 
\eprf

\section{Cones with parallel null $2$-planes}
\label{planesec}
In this section we consider the base manifolds $(M,g)$ of cones that admit a parallel distribution of totally null $2$-planes. 
Our main result is the description of the most general local form of the metric $g$. 
To exclude trivial cases we assume $\dim M >1$. 

\subsection{The induced structure on the base}
If $(\hm,\hg)$ is a semi-Riemannian manifold and  $\widehat{\mathbf{P}}$  a parallel totally null $2$-plane bundle, then    locally   there are two null vector fields $\chi$ and $\zeta$ that are orthogonal to each other and such that
\belabel{pplane}
\begin{array}{rcl}
\hnab\chi &=& \alpha \otimes\chi+\mu\otimes\zeta,
\\
\hnab\zeta&=&\beta \otimes \chi+\nu\otimes \zeta,
\end{array}
\eeq
for $1$-forms $\alpha$, $\beta$, $\mu$ and $\nu$.

If $(\hm,\hg)$ is a timelike cone  with a parallel null $2$-plane  bundle $\widehat{\mathbf{P}}$,  we can intersect $\widehat{\mathbf{P}}$ with $\xi^\perp$, where $\xi$ is the Euler vector field. 
A subset of $\hm = \mathbb{R}^{>0}\times M$ will be called \emph{conical} if it is of the form
$\widehat{M}_0=\mathbb{R}^{>0}\times M_0$ for some subset $M_0\subset M$.
\blem\label{hilf:lem}
On  a conical open dense subset in $\hm$
the intersection $\widehat{\mathbf{P}}\cap \xi^\perp$  is a null-line bundle $\mathbf{L}$
invariant under the flow of $\xi$. In particular, ${\mathbf{L}}$ admits local sections, defined on conical open sets,  
invariant under the flow of $\xi$ and descends to a null line distribution on an open dense subset of $\M$. 
\elem
\bprf
For this and the following proofs, we note that
\[ \left[\xi , \Gamma(\xi^\perp)\right]\subset \Gamma(\xi^\perp)\quad\mbox{and}\quad \left[\xi , \Gamma(\widehat{\mathbf{P}})\right]\subset \Gamma(\widehat{\mathbf{P}}).\]
This implies that the dimension of the fibres of $\widehat{\mathbf{P}}\cap \xi^\perp$ is constant on the integral curves of $\xi$. 
At each point $p\in \hm$, $\xi^\perp|_p$ is a hyperplane and $\widehat{\mathbf{P}}|_p$ a $2$-plane in $T_p\hm$. Hence their intersection has dimension one or two. Now let us  assume that, over an  open set $U\subset \hm$, $\widehat{\mathbf{P}}\cap \xi^\perp$ is of rank $2$, i.e.\ that $\widehat{\mathbf{P}}\subset \xi^\perp$. Hence $\widehat{\mathbf{P}}\cap \xi^\perp$ a distribution of $2$-planes spanned by vector fields $V_1$ and $V_2$ on $U$  that are tangential to $\M$. Then formulae \re{coneLC} and \re{pplane}  give us
\[
T\M\ni \widehat\nabla_XV_i = \nabla_XV_i +\gg(X,V_i)\xi,
\]
for all $X\in T\M$. Hence, on $U$ it is $\gg(X,V_i)=0$ for all $X\in T\M$ which is impossible. Consequently, 
the conical open set over which the fibres of $\widehat{\mathbf{P}}\cap \xi^\perp$ are one-dimensional is dense and 
$\widehat{\mathbf{P}}\cap \xi^\perp$ restricts to a line bundle $\mathbf{L}$ over that set. 
\eprf
Now  we  project $\widehat{\mathbf{P}}$ to $\xi^\perp$. 
\blem
The projection $\mathrm{pr}_{\xi^\perp}(\widehat{\mathbf{P}})\subset \xi^\perp$ is an involutive $2$-plane distribution $\mathbf{P}$ on $\hm$ and descends to an involutive  $2$-plane distribution on $\M_0$.
\elem
\bprf
First note that  the fibres of $\mathrm{pr}_{\xi^\perp}(\widehat{\mathbf{P}})$ have dimension $2$ because $\widehat{\mathbf{P}}\cap \mathbb{R}\cdot \xi=\{0\}$. Hence,  $\mathbf{P}:=\mathrm{pr}_{\xi^\perp}(\widehat{\mathbf{P}})\subset \xi^\perp$ is a $2$-plane distribution.

Clearly the projection of a vector field $V$ on $\hm$ to $\xi^\perp$ is given as
\[\pr_{\xi^\perp}(V)=V+r^{-2}\widehat\gg(V,\xi)\xi.
\]
By a calculation using $\hnab\xi=\Id$  we obtain for all $V_1, V_2 \in \mathfrak{X}(\hm)$:
\be
\left[\pr_{\xi^\perp}(V_1),
\pr_{\xi^\perp}(V_2)
\right]
&=&
\pr_{\xi^\perp}\left( [V_1,V_2]+ r^{-2}\widehat{g}(V_2,\xi )[V_1,\xi] -r^{-2}\widehat{g}(V_1,\xi )[V_2,\xi]\right).
\ee
Since the  distribution $\widehat{\mathbf{P}}$ is invariant under $\xi$, parallel and hence involutive, the right-hand side
is a section of $\mathbf{P}$ for all sections $V_1, V_2$ of $\widehat{\mathbf{P}}$. This proves the involutivity of
$\mathbf{P}$. The distribution $  \mathbf{P}$ descends to $M$ due to the invariance under $\xi$.\eprf
Moreover we obtain:
\blem\label{VZprop}
There exist local sections $V$ of $\mathbf{L}$ and $Z$ of $\mathbf{P}$, defined on a 
conical open set, 
such that $V$ and 
 \be
 \zeta&=&\xi + Z
 \ee
locally span $\widehat{\mathbf{P}}$ and satisfy 
 \be
 [\xi, V]=0&\text{and}& [\xi,Z]=0.
 \ee
The vector fields $V$ and $Z$ descend to local vector fields on  $\M$.
 \elem
 \bprf
 We have already seen that there exists a non-vanishing section $V$ of $\mathbf{L}$  
 over a 
conical open set such that $[\xi,V] =0$. In the following we always work locally over conical
open sets. 
Every section of $\widehat{\mathbf{P}}$ that is nowhere a multiple of $V$ is of the form $f\xi+Z$ for $Z$ a (possibly vanishing) local section of $\mathbf{P}$ and $f$ a non-vanishing local function on $\hm$. Hence, by multiplying with $1/f$ we can assume that we have a section
 \[\hat\zeta=\xi+ \hat Z\]
 of $\widehat{\mathbf{P}}$. 
We will now use the freedom to add multiples of $V$ to $\hat Z$ without leaving $\widehat{\mathbf{P}}$, in order to find a $Z=\hat Z+\vf V$ for which we have $[\xi,Z]=0$. Indeed, writing 
\[\nabla_\xi\hat\zeta= f V + h \hat\zeta
\]
with functions $f$ and $h$, we compute
\[
[\xi,\hat Z]
= 
[\xi,\hat \zeta]
=
f V+ (h-1)\hat\zeta .
\]
Since $[\xi,\hat Z]$ belongs to $\xi^\perp$, we must have that $h\equiv 1$ and
\[
[\xi,\hat Z]=f V.\]
Now if we  fix a solution $\vf$ of
\[\d\vf(\xi)+f=0,\]
and set $Z=\hat Z+\vf V$ we get
\[[\xi,Z]=0.\]
Clearly, since $V$ is a section of $\widehat{\mathbf{P}}$, the vector field 
\[\zeta:=\xi +Z =\hat\zeta+\vf V,\]
is also a section in $\widehat{\mathbf{P}}$ that is still linearly independent of $V$ and therefore $Z$ is a section of $\mathbf{P}$ that locally descends to $\M$. \eprf

\btheo
Let $(\hm,\hg)$ be a timelike cone over a semi-Riemannian manifold $(\M,\gg)$. 
If the  cone admits  a parallel distribution of  totally  null $2$-planes field, then 
the base $(\M,\gg)$  admits locally two vector fields $V$ and $Z$ such that
\begin{equation}\label{gVZ}
\gg(V,V)=0,\ \ \gg(Z,Z)=1,\ \ \gg(V,Z)=0,\end{equation}
and 
 \begin{eqnarray}
  \label{nabV}
  \nabla_XV
  &
  =
  &\alpha (X)V + \gg(X,V)Z, 
  \\
  \label{nabZ}
    \nabla_XZ
  &=&
  -X +\beta(X)V+\gg(X,Z)Z,
  \end{eqnarray}
   for all $X\in T\M$, with $1$-forms $\alpha $ and $\beta$ on $\M$.
   
 Conversely, each pair of vector fields $V$ and $Z$ on $M$ satisfying relations (\ref{gVZ}), (\ref{nabV}) and (\ref{nabZ}) defines a parallel distribution of totally null $2$-planes on the cone.
 \etheo
  \bprf
First assume that the cone admits a parallel totally null $2$-plane $\widehat{\mathbf{P}}$ which is spanned by 
$V$ and $\zeta=\xi+Z$ as in \Cref{VZprop}. Equations \re{gVZ} are implied by $\widehat{\mathbf{P}}$ being totally null. Moreover, 
  equations \re{pplane} with $\chi =V$ and $X\in T\M$  become
  \begin{eqnarray} \label{nabV1}
  \hnab_XV\; =\; \nabla_XV+\gg(X,V)\xi
  &=&\alpha (X)V + \mu(X)(\xi+Z),
  \\  \label{nabZ1}
   \hnab_X\zeta\;=\;X+\nabla_XZ+\gg(X,Z)\xi&=&\beta(X)V+\nu(X)(\xi+Z), 
  \end{eqnarray}
 and imply
 \be
 \mu(X)&=&\gg(X,V),
 \\
  \nu(X)&=&\gg(X,Z),
  \ee
  as well as equations \re{nabV} and \re{nabZ}, but still with $r$-dependent $1$-forms $\alpha$ and $\beta$. Hence, it remains to show that $\alpha$ and $\beta$, when restricted to $\xi^\perp$, are invariant under the flow of $\xi$ and therefore descend to $1$-forms on $\M$, i.e., that
 \[
 \cal L_\xi\alpha|_{\xi^\perp}=\cal L_\xi\beta|_{\xi^\perp}=0.\]
 But  from
 \be
 0&=&
 \hR(\xi,X)V
 \\
 &=&(\mathcal{L}_\xi\alpha )(X) V+ \alpha(X) V+ \gg(X,V)(\xi+Z)
 -
 \left(\nabla_XV+\gg(X,V)\xi\right)\\
 &=& (\mathcal{L}_\xi\alpha )(X) V,
 \ee
 because of equation 
 \re{nabV1}. This proves that $\mathcal{L}_\xi\alpha|_{\xi^\perp}=0$. 
 Analogously we get
  \[ 
 0 =
 \hR(\xi,X)\zeta =(\mathcal{L}_\xi\beta )(X) V \]
 and again $\mathcal{L}_\xi\beta|_{\xi^\perp}=0$. 
  
 Conversely, if we start with a manifold $(\M,\gg)$ and vector fields satisfying conditions \re{gVZ}, \re{nabV} and \re{nabZ}, a straightforward computations shows that the cone admits  a parallel null plane spanned by $V$ and $\xi+Z$.
 \eprf
 \bfolg
 If the cone \re{conemetric} admits  a distribution of parallel totally null $2$-planes, then the base $(\M,\gg)$ admits locally a geodesic, shearfree null vector field $V$.
 \efolg
  \bprf
Since $V$ is null, equation \re{nabV} implies that $V$ is geodesic.
Recall that a geodesic  null vector field is called {\em shearfree} if 
\[\cal L_V\gg =\lambda\gg +\theta\cdot V^\flat,
\]with a function $\lambda$ and a $1$-form $\theta$ and where the dot stands for the symmetric product.
From \re{nabV} and the formula 
\begin{equation} \label{LXg:eq} \mathcal{L}_Xg= 2 (\nabla X^\flat)^{\mathrm{sym}},\end{equation} 
where `sym' denotes the projection onto
the symmetric part, we compute
\[
\cal L_V\gg =2(\alpha+Z^\flat)\cdot V^\flat,
\]
i.e., the shear free condition is satisfied with $\lambda=0$.	
  \eprf
\bbem  \label{trafo:rem} We can change the basis of $\mathrm{span}(V,Z)$ to $V', Z'$ such that $V'$ is still null and orthogonal to $Z'$ and such that $Z'$ is a unit vector field,
\[(V,Z)\longmapsto (V'=\e^{f}V, Z'=Z +h V).\] 
Then the $1$-forms $\alpha$ and $\beta $ transform as
\begin{eqnarray*}
\alpha&\longmapsto& \alpha'= \alpha+ \d f -hV^\flat,
\\
\beta&\longmapsto& \beta'= e^{-f}(\beta + h \alpha+ \d h -hZ^\flat -h^2V^\flat ).
\end{eqnarray*}
\ebem

\subsection{Consequences of the fundamental equations}
Let $(M,g)$ be a semi-Riemannian manifold endowed with two pointwise linearly independent vector fields 
$V$, $Z$ which satisfy \re{gVZ}, \re{nabV} and \re{nabZ}. 
\bs \label{fundeq:prop}The fundamental equations \re{gVZ}\, \re{nabV} and \re{nabZ} imply
\begin{eqnarray} d V^\flat &=&(\alpha -Z^\flat )\wedge V^\flat,\label{dVb:equ}\\
d Z^\flat &=& \beta \wedge V^\flat ,\label{dZb:equ}\\
{[}Z,V{]} &=& (\alpha (Z) -\beta (V) +1) V \label{ZV:equ},\\
 \cal L_Vg &=& {2(\alpha + Z^\flat )V^\flat},\label{LVg:equ}\\
 \cal L_Zg &=& {-2g + 2(Z^\flat )^2 + 2\beta V^\flat}\label{LZg:equ},   
\end{eqnarray}
where  we are using the symmetric product of $1$-forms in the last two formulas. 
\es
\bprf
Since $\nabla$ is torsion-free, the differential of any 1-form $\varphi$ is 
given by 
\[ \d\varphi (X,Y) = (\nabla_X\varphi ) Y-(\nabla_Y\varphi )X,\quad X,Y\in \mathfrak{X}(M).\]
Now \re{dVb:equ} and \re{dZb:equ} follow immediately from \re{nabV} and \re{nabZ}. 
Using again that $\nabla$ is torsion-free, the fundamental equations  easily imply \re{ZV:equ}.  
Similarly, the last two formulas follow from \re{LXg:eq}. 
\eprf 

\bfolg \label{cor78}We have
\begin{eqnarray} 
 \mathcal{L}_VV^\flat &=& \alpha (V)V^\flat ,\\
 \mathcal{L}_ZV^\flat &=& 
{
 \left(  \alpha(Z)-1\right) V^\flat},
 \\
  \mathcal{L}_VZ^\flat &=&
  {
  \beta(V) V^\flat}.
\end{eqnarray} 
The vector fields $Z$ and $V$ commute if and only if 
\begin{equation}
{\beta (V)}= \alpha (Z) +1.
  \label{ab:eq}
  \end{equation}
\efolg
\bprf The first three formulas are obtained from Cartan's formula for the Lie derivative to the equations (\ref{dVb:equ}) and (\ref{dZb:equ}).   Alternatively one can use (\ref{ZV:equ}), (\ref{LVg:equ}) and (\ref{LZg:equ}). 
The last assertion follows from equation (\ref{ZV:equ}).
\eprf 

\bfolg  \label{cor79}
By multiplying $V$ with a function
we can locally assume that 
\begin{equation} \d V^\flat =0,\label{grad:equ}\end{equation} 
that is 
\be \alpha = Z^\flat + f_\alpha V^\flat\ee 
for some function $f_\alpha$. 
The latter equation implies 
\be \alpha (Z)=1,\quad \alpha (V)=0, \quad  \mathcal{L}_VV^\flat =0, \quad \mathcal{L}_ZV^\flat =0 .\ee 
By adding a functional multiple of $V$ to $Z$ we can further locally assume that 
\be \beta (V)=2,\ee 
which implies $\mathcal{L}_VZ^\flat = 2V^\flat
$ and 
is equivalent to 
$ [Z,V]=0$.
\efolg 

\bprf By equation (\ref{dVb:equ}) and the Frobenius theorem,  the hyperplane distribution $V^\perp$ is integrable, which locally  implies
that a functional multiple of $V^\b$ is closed. The equations and the second statement follow from the transformation formulae for $\alpha$ and $\beta$  in Remark \ref{trafo:rem} and Corollary \ref{cor78}.
\eprf

\bfolg\label{refcor}
With the normalisation 
that $dV^\flat=0$,
 the leaves of the  integrable distribution $V^\perp$ are totally geodesic and the vector field $V$ preserves the tensor field $g|_{V^\perp\times V^\perp}$.
\efolg
\bprf 
For $X,Y\in V^\perp$ we have \[g(\nabla_XY,V) =-g(Y,\nabla_X V),\]
and because of $dV^\flat=0$,
\[
g(Y,\nabla_X V)=\frac12 (\cal{L}_Vg)(X,Y).\]
Using equation (\ref{LVg:equ}) for $X,Y\in V^\perp$ we get  $\cal L_Vg)(X,Y)=0$ and hence $g(\nabla_XY,V)=0$, which means that the leaves of $V^\perp$ are totally geodesic.
\eprf

\subsection{The local form of the metric on the base}\label{plane-local-form-sec}
In the following we will assume all of the above equations. By \re{grad:equ}, locally, there exists a function $u$ such that
$du=V^\flat$. The function $u$ is constant on each leaf $L$ of the distribution $V^\perp$. Locally, we can decompose
$M$ as $M=L\times \rr$, such that $u$ corresponds to the coordinate on the $\rr$-factor and the leafs of $V^\perp$ are
the hypersurfaces $L_u=L\times \{ u\}$. Since the vector fields $V$ and $Z$ commute and are tangent to $V^\perp$, we can further
decompose each leaf of $V^\perp$ locally as $L_u\cong L=M_0\times \rr \times \rr$, such that $V=\partial_t$, $Z=\partial_s$ are the coordinate vector fields
tangent to the first and second $\rr$-factor, respectively.  

Let us denote by $\mathbf{P}$ the integrable distribution spanned by $V$ and $Z$. Notice that  by  \re{dZb:equ} the distribution $\mathbf{P}^\perp=Z^\perp \cap V^\perp$ is also integrable, in virtue of the Frobenius theorem. So we can assume that the level sets of $s$ are tangent to $\mathbf P^\perp$. Finally,  
the decomposition $M=L\times \rr$ can be chosen such that the decomposition 
$L_u = M_0 \times \rr \times \rr$ is independent of $u$, that is the vector field $\partial_u$ commutes with $V$, $Z$ and with the 
canonical lift of vector fields of $M_0$.  

\btheo Let $(M,g)$ be a semi-Riemannian manifold  such that the cone $(\widehat{M} , \widehat{g})$ admits
a parallel totally null distribution of $2$-planes. In terms of the above local decomposition $M = M_0 \times \rr^3$ we have
\begin{equation} g = \d s^2 + e^{-2s}g_0(u) + 2\,\d u\,\eta,\label{metric:equ}\end{equation}
for some 1-form $\eta$ on $M$ such that $\eta (\partial_t)$ is nowhere vanishing and a family of metrics $g_0(u)$ on $M_0$ depending
on $u$. 
\etheo 

\bprf The restriction of the metric
to a leaf $N=M_0\times \rr \times \{ (s,u) \}$ of $\mathbf P^\perp$ is degenerate with kernel $V=\partial_t\in \mathbf{P}^\perp$ and invariant under the flow
of $V$, see \re{LVg:equ}. Since $M_0$ is transversal to $V$, we see that 
$g|_N=g_0(u,s)$ for some family of metrics on $M_0$  depending on $u$ and $s$. The flow of $Z=\partial_s$ is a 1-parameter family of homotheties of weight $-2$, 
see \re{LZg:equ}. This shows that $g_0(u,s)=e^{-2s}g_0(u)$ for some 1-parameter family of metrics $g_0(u)$. It follows that on the leafs $L_u = M_0 \times \rr \times \rr \times \{ u\}$ of $V^\perp$ the metric is of the form $\d s^2 + e^{-2s}g_0(u)$. Finally, on $M$ we obtain the general form \re{metric:equ} with $\eta (\partial_t)\neq 0$, in view of the 
non-degeneracy of $g$. 
\eprf 
It remains to determine the necessary and sufficient conditions for the data $g_0(u)$ and $\eta$ ensuring that the 
cone over $(M,g)$ as in (\ref{metric:equ}) admits a parallel totally null distribution of $2$-planes. Let $M_0$ be a manifold and let us denote the standard coordinates on $\rr^3$ by $(t,s,u)$.  
\btheo \label{plane-local-form}
For any 1-form $\eta$ on $M:= M_0\times \rr^3$ such that $\eta_t:=\eta (\partial_t)\neq 0$ and any family of semi-Riemannian 
metrics $g_0(u)$ on $M_0$ the tensor field 
\[ g= \d s^2 + e^{-2s}g_0(u) + 2\,du\,\eta,\] 
cf.\ \re{metric:equ},  is a semi-Riemannian metric on $M$ such that the vector fields $V=\partial_t$ and 
$Z=\partial_s$ satisfy \re{gVZ}. The covariant derivatives of $V$ and $Z$ are given by  \re{nabV} and \re{nabZ} 
for some 1-forms $\alpha =Z^\flat + f_\alpha V^\flat$ and $\beta$ 
such that  $f_\alpha$ is a function on $M$ and $\beta (V)=2$, 
if and only if 
the coefficients of $\eta$ 
solve the following system of first order partial differential equations: 
 \begin{equation} \partial_t\eta_t=\partial_s\eta_t=X\eta_t=\partial_t\eta(X)=0,\quad \partial_t\eta_s = 2\eta_t,\quad\partial_s\eta (X) -X\eta_s=-2\eta (X) \label{system:equ}\end{equation}
 for all $X\in \mathfrak{X}(M_0)$. Then  $\alpha$ and $\beta$ are determined by 
\begin{eqnarray*} &&f_\alpha = \frac{1}{\eta_t^2} \partial_t\eta_u - \frac{2}{\eta_t}\eta_s,\quad \beta (Z) = \frac{1}{\eta_t}\partial_s\eta_s,\quad 
 \beta (X) = \frac{1}{2\eta_t}(X\eta_s +\partial_s\eta (X) +2 \eta (X)),\\
 && \beta (\partial_u ) = \frac{1}{\eta_t}(\partial_s\eta_u -\eta_s^2+2\eta_u).
\end{eqnarray*} 
\etheo 

\bprf We denote by $X$ the canonical lift of a vector field on $M_0$. Then $X, V, Z$ and $\partial_u$ commute and using the Koszul formula we obtain
\begin{eqnarray*} 
&&g(\nabla_VV,X)=g(\nabla_VV,V)=g(\nabla_VV,Z)=0,\quad g(\nabla_VV,\partial_u ) = \partial_t\eta_t, \\
&&g(\nabla_ZV,X)=g(\nabla_ZV,V)=g(\nabla_ZV,Z)=0,\quad 2g(\nabla_ZV,\partial_u ) = \partial_s\eta_t+\partial_t\eta_s,\\
&&g(\nabla_XV,X)= g(\nabla_XV,V)=g(\nabla_XV,Z)=0,\quad 2g(\nabla_XV,\partial_u ) = X\eta_t + \partial_t\eta (X),\\
&&2g(\nabla_{\partial_u}V,X)= \partial_t\eta(X)-X\eta_t,\quad  g(\nabla_{\partial_u}V,V)=0,\quad  2g(\nabla_{\partial_u}V,Z)=\partial_t\eta_s-\partial_s\eta_t,\\ 
&&g(\nabla_{\partial_u}V,\partial_u ) = \partial_t\eta_u,\\
&& g(\nabla_VZ,X) = g(\nabla_VZ,V)=g(\nabla_VZ,Z)=0,\quad 2g(\nabla_VZ,\partial_u) = \partial_t\eta_s+\partial_s\eta_t,\\
&& g(\nabla_ZZ,X)=g(\nabla_ZZ,V)=g(\nabla_ZZ,Z)=0,\quad g(\nabla_ZZ,\partial_u)= \partial_s\eta_s,\\
&& g(\nabla_XZ,X)=-g(X,X),\quad g(\nabla_XZ,V)=g(\nabla_XZ,Z)=0,\\ 
&&2g(\nabla_XZ,\partial_u) 
= X\eta_s +\partial_s\eta (X),\\
&& 2g(\nabla_{\partial_u}Z,X)=\partial_s\eta (X) -X\eta_s,\quad 2g(\nabla_{\partial_u}Z,V)=\partial_s\eta_t-\partial_t\eta_s,\quad g(\nabla_{\partial_u}Z,Z)=0,\\ 
&&g(\nabla_{\partial_u}Z,\partial_u) = \partial_s\eta_u. 
\end{eqnarray*}
Comparing with \re{nabV}, \re{nabZ} we obtain the above formulas for $\alpha$ and $\beta$ and the following system for $\eta$:
\begin{eqnarray*}&& \partial_t\eta_t = 0,\quad \partial_s\eta_t+\partial_t\eta_s = 2 \eta_t,\quad X\eta_t + \partial_t\eta (X) = 0,\quad \partial_t\eta(X)-X\eta_t=0,\\
&&\partial_t\eta_s-\partial_s\eta_t=2\eta_t,\\ 
&&\partial_s\eta (X) -X\eta_s=-2\eta (X) 
\end{eqnarray*} 
for all $X\in \mathfrak{X}(M_0)$. 
This system can be brought to the form \re{system:equ}.
\eprf
For convenience we denote a system of local
coordinates on $M_0$ by $(x^i)_{i=1,\ldots ,n_0}$ and denote by $x$ the corresponding coordinate vector, where $n_0=\dim M_0$. The general solution of \re{system:equ} is obtained as follows.
\bs  Let $f_1=f_1(u)$ be an arbitrary nowhere vanishing smooth function on 
the real line equipped with the coordinate $u$ and $f_2= f_2(x,s,u)$ an arbitrary smooth function on $M$ which does not 
depend on $t$. Let $h_i=h_i(x,s,u)$ be a ($t$-independent) solution of the ordinary differential equation 
\[ \partial_sh_i +2h_i= \partial_if_2\]
for all $i=1,\ldots ,n_0$, where $\partial_i = \partial /\partial x^i$. Then 
\[ \eta_t := f_1(u),\quad \eta_s :=2tf_1(u)+f_2(x,s,u),\quad \eta (\partial_i) := h_i(x,s,u)\]
solves  \re{system:equ} and every solution is of this form. 
\es

\bbem
Finally we return to the Lorentzian metrics that occurred in \Cref{holtheointro} and arose from the case where the cone $(\hm,\hg)$ admits a parallel null line: in this case the cone metric $\hg$ was isometric to the metric $\tg=2 \d u\d v+u^2 g_0$ with a Lorentzian metric $g_0$ and $g$ was isometric to $g=\d s^2+\mathrm{e}^{2s} g_0$. Then \Cref{holtheointro} stated that if 
 the holonomy of the cone is not equal to $\hol(g_0)\ltimes \rr^{1,n-1}$, then
 $g_0$ admits a parallel null vector field. It is well known (see for example \cite{schimming74,galaev-leistner09}) that   locally $g_0$ is of the form
 $g_0=2\d x\d z+h(z)$, where $h(z)$ is a $z$-dependent family of Riemannian metrics. Hence, $g$ is of the form
\[
g=ds^2+\mathrm{e}^{2s} h(z)+2\mathrm{e}^{2s}\d x\d z.\]
This corresponds to the local form in \Cref{plane-local-form}, where $x$ corresponds to $t$ and $2\mathrm{e}^{2s}\d x$ to $\eta$, $z$ to $u$ and $h(z)$ to $g_0(u) $. 

 
\ebem

\bibliographystyle{abbrv}

\providecommand{\MR}[1]{}\def\cprime{$'$} \def\cprime{$'$} \def\cprime{$'$}

\end{document}